\documentclass[10pt,twoside]{amsart}

\usepackage{amssymb}
\usepackage{amscd}
\usepackage{amsmath}
\usepackage{amsrefs}
\usepackage{xypic}
\usepackage[OT2,T1]{fontenc}
\usepackage[mathscr]{euscript}
\usepackage{hyperref}
\usepackage{color}
\hypersetup{
    colorlinks=true,
    citecolor=black,
    linkcolor=black,
  urlcolor=blue,
}

\DeclareSymbolFont{cyrletters}{OT2}{wncyr}{m}{n}
\DeclareMathSymbol{\Sha}{\mathalpha}{cyrletters}{"58}

\newtheorem{theorem}{Theorem}[section]
\newtheorem{prop}{Proposition}[section]
\newtheorem{lemma}{Lemma}
\newtheorem{corollary}{Corollary}

\newtheorem{definition}{Definition}
\theoremstyle{remark}
\newtheorem{remark}{Remark}

\newcommand{\ba}{\mathbb A}
\newcommand{\bq}{\mathbb Q}

\newcommand{\bz}{\mathbb Z}
\newcommand{\br}{\mathbb R}
\newcommand{\bc}{\mathbb C}
\newcommand{\bg}{\mathbb G}

\newcommand{\bF}{\mathbb F}

\newcommand{\bl}{\mathbb L}

\newcommand{\co}{\mathcal O}
\newcommand{\cd}{\mathcal D}

\newcommand{\ce}{\mathcal E}
\newcommand{\ca}{\mathcal A}
\newcommand{\cb}{\mathcal B}
\newcommand{\csp}{\mathsf p}

\newcommand{\fa}{\mathfrak a}
\newcommand{\fp}{\mathfrak p}
\newcommand{\fq}{\mathfrak q}
\newcommand{\ff}{\mathfrak f}
\newcommand{\fm}{\mathfrak m}
\newcommand{\fl}{\mathfrak l}
\newcommand{\fg}{\mathfrak g}

\newcommand{\X}{\mathcal X}

\newcommand{\mydet}{\mathrm{det}}

\DeclareMathOperator{\Frob}{Fr}
\DeclareMathOperator{\ord}{ord}
\DeclareMathOperator{\Spec}{Spec}
\DeclareMathOperator{\rank}{rank}

\DeclareMathOperator{\trace}{Tr}
\DeclareMathOperator{\Pic}{Pic}

\DeclareMathOperator{\Ind}{Ind}
\DeclareMathOperator{\Res}{Res}

\DeclareMathOperator{\Br}{Br}

\DeclareMathOperator{\Gal}{Gal}

\DeclareMathOperator{\id}{id}
\DeclareMathOperator{\Hom}{Hom}
\DeclareMathOperator{\End}{End}
\DeclareMathOperator{\Ext}{Ext}
\DeclareMathOperator{\Aut}{Aut}

\DeclareMathOperator{\Lie}{Lie}

\begin{document}

\title[The conjecture of Birch and Swinnerton-Dyer ]{The conjecture of Birch and Swinnerton-Dyer for certain elliptic curves with complex multiplication}
\author{Ashay Burungale}
\author{Matthias Flach}
\dedicatory{Dedicated to John H. Coates}
\address{Dept. of Mathematics 258\\California Institute of Technology\\Pasadena CA 91125\\ And
The University of Texas at Austin, Austin, TX 78712, USA.}
\email{ashayburungale@gmail.com}
\address{Dept. of Mathematics 253-37\\California Institute of Technology\\Pasadena CA 91125\\USA}
\email{flach@caltech.edu}
\begin{abstract} Let $E/F$ be an elliptic curve over a number field $F$ with complex multiplication by the ring of integers in an imaginary quadratic field $K$. We give a complete proof of the conjecture of Birch and Swinnerton-Dyer for $E/F$, as well as its equivariant refinement formulated by Gross \cite{gross2}, under the assumption that $L(E/F,1)\neq 0$ and that $F(E_{tors})/K$ is abelian. We also prove analogous results for CM abelian varieties $A/K$. 
\end{abstract}
\maketitle
\setcounter{tocdepth}{1}
\tableofcontents
\section{Introduction}
Let $E/F$ be an elliptic curve over a number field $F$ with complex multiplication by the ring of integers in an imaginary quadratic field $K$ and such that $F(E_{tors})/K$ is abelian. 
It is well known that 
the conjecture of Birch and Swinnerton-Dyer for this class of elliptic curves, as well as its $K$-equivariant refinement formulated by Gross \cite{gross2}, is amenable to the Iwasawa theory of the field $K$. 
Indeed, this principle has its origin in the seminal work of Coates and Wiles \cite{coates77} which led to the finiteness of $E(F)$ if $L(E/F,1)\neq 0$. 
About a decade later Rubin \cite{rubin87} showed  the  finiteness of $\Sha(E/F)$ if $L(E/F,1)\neq0$. This remarkable work, partly motivated by the ideas of \cite{thaine88}, gave the very first proof of finiteness of the Tate-Shafarevich group of an abelian variety over a number field. Subsequently, as a consequence of his proof of the Iwasawa main conjecture for $K$, Rubin \cite{rubin91} proved the $\fp$-primary part of Gross' conjecture assuming $F=K$, $L(E/F,1)\neq 0$ and $\fp\nmid |\co_K^\times|$. He also indicated that for general $F$ his arguments give a proof of the $\fp$-primary part of Gross' conjecture if $L(E/F,1)\neq 0$ and 
$$
\fp\nmid |\co_K^\times|\cdot[F:K]\cdot\mathrm{disc}(F/K).
$$

The purpose of this paper is to eliminate these restrictions on the prime ideal $\fp$ of $\co_K$ and give a complete proof of Gross' conjecture if $L(E/F,1)\neq 0$. 
The main result is Theorem \ref{main} below.
Our approach is based on the principle of the equivariant Tamagawa number conjecture: zeta elements generate equivariant determinants of certain \'etale cohomology groups.
The key ingredients of the proof are the two variable Iwasawa main conjecture for $K$ due to Johnson-Leung and Kings \cite{jlk} (based on the Euler system of elliptic units) and 
Kato's reciprocity law \cite{kato00}[Prop. 15.9]  (we use Kato's formulation but the case we need goes back to Wiles \cite{wiles78} and Coates/Wiles \cite{cw78}). 
Our arguments also give the $L$-equivariant Birch and Swinnerton-Dyer conjecture for abelian varieties $A/K$ with complex multiplication by a CM field $L$ if $L(A/K,1)\neq 0$ which we record in Theorem \ref{main2}. In particular this proves a conjecture of Buhler and Gross \cite{bugross85}[Conj. 12.3]. 

We first introduce some notation. For archimedean places $v$ of $F$ we have the $K$-bilinear integration pairings
\begin{equation} H_1(E(F_v),\bq) \times H^0(E,\Omega_{E/F})\otimes_FF_v \to F_v;\quad  (\gamma,\omega)\mapsto \int_\gamma\omega\label{integration}\end{equation}
which jointly induce a $K_\br$-linear period isomorphism
\begin{equation} \prod_{v\mid\infty} H_1(E(F_v),\bq)_\br\cong \Hom_{F}(H^0(E,\Omega_{E/F}),F)_\br.\label{period}\end{equation} 
For each $v\mid \infty$ the period lattice $H_1(E(F_v),\bz)$ is an invertible $\co_K$-module. If $\ce/\co_F$ denotes the N\'eron model of $E$ then $H^0(\ce,\Omega_{\ce/\co_F})$ is an invertible $\co_F$-module and a projective $\co_K$-module of rank $d=[F:K]$, hence so is
$\Hom_{\co_F}(H^0(\ce,\Omega_{\ce/\co_F}) ,\co_F)$. It follows that there is an invertible $\co_K$-submodule $\fa\subset K_\br$ so that
$$\bigotimes\limits_{v\mid\infty} H_1(E(F_v),\bz)=\fa\cdot\mydet_{\co_K}\Hom_{\co_F}(H^0(\ce,\Omega_{\ce/\co_F}) ,\co_F)$$
under the determinant over $K_\br$ of the isomorphism (\ref{period}). We may write 
$$\fa=\Omega\cdot\fa(\Omega)$$
for some period 
$$\Omega\in K_\br^\times\cong\bc^\times$$ 
and fractional $\co_K$-ideal $\fa(\Omega)\subseteq K$. Denote the order ideal of a finite $\co_K$-module $A$ by $|A|_K$ and the cardinality of a finite abelian group $A$ by $|A|$. Let $\Phi_v$ be the component group of the N\'eron model of $E/F$ at the prime $v$.

\begin{theorem} Let $E/F$ be an elliptic curve over a number field $F$ with CM by $\co_K$ for an imaginary quadratic field $K$ and such that $F(E_{tors})/K$ is abelian. Let $\psi:\ba_F^\times/F^\times\to \bc^\times$ be the Hecke character associated to $E/F$ and assume that $L(\overline{\psi},1)\neq 0$. Then $E(F)$ and $\Sha(E/F)$ are finite $\co_K$-modules,
$$ \frac{L(\overline{\psi},1)}{\Omega}\in K^\times$$
and 
$$ \frac{L(\overline{\psi},1)}{\Omega}=\frac{|\Sha(E/F)|_K}{|E(F)|}\cdot \prod_v|\Phi_v|_K\cdot \fa(\Omega) $$
in the group of fractional $\co_K$-ideals. 
\label{main}\end{theorem}
 \begin{remark}  
 As pointed out by Gross, not only the ideal $|E(F)|$ but also the ideal $|\Sha(E/F)|_K$ is generated by a rational integer \cite{gross2}[Prop. 3.7]. The ideals $|\Phi_v|_K$ are equal to either $(1)$, $(2)$ or $\fp$ with $\fp^2=(2)$ or $\fp^2=(3)$ \cite{gross2}[Prop. 4.5].
\end{remark}

Restricting scalars from $K_\br$ to $\br$ in the period isomorphism (\ref{period}) and taking determinants over $\bz$ of the natural lattices in both sides, we find that there exists $\Omega(E)\in\br^\times$ such that
$$\bigotimes_{v\mid\infty}  \mydet_\bz H_1(E(F_v),\bz)=\Omega(E)\cdot \mydet_\bz \Hom_{\co_F}(H^0(\ce,\Omega_{\ce/\co_F}) ,\co_F).$$
Moreover, we have
\begin{equation} \Omega(E)\cdot\bz= N_{K/\bq}\fa =\fa\overline{\fa}=\Omega\overline{\Omega}\cdot\fa(\Omega)\overline{\fa(\Omega)}.\label{periodnorm}\end{equation}

\begin{corollary}  (BSD for $E/F$) Under the assumptions of Theorem \ref{main} the groups $E(F)$ and $\Sha(E/F)$ are finite and
$$ \frac{L(E/F,1)}{\Omega(E)}=\frac{|\Sha(E/F)|}{|E(F)|^2}\cdot \prod_v|\Phi_v|$$
\label{BSD1}\end{corollary}

\begin{proof} This follows from the identity \cite{shimurabook}[Thm. 7.42]
$$ L(E/F,s)=L(\overline{\psi},s)\overline{L(\overline{\psi},s)}=L(\overline{\psi},s)L(\psi,s)$$
the fact that $N_{K/\bq}(|A|_K)=|A|$ for any finite $\co_K$-module $A$, and (\ref{periodnorm}).
\end{proof}
\noindent Any elliptic curve $E/K$ with CM by $\co_K$ for which $L(E/K,1)\neq 0$ satisfies the assumptions of Theorem \ref{main} and Corollary \ref{BSD1}.  In this case the class number of $K$ is $1$. More generally, for primes $q\equiv 3 \mod{4}$ and $K=\bq(\sqrt{-q})$ the class number of $K$ is odd and elliptic curves  $E/H$ where $F=H$ is the Hilbert class field have been much studied in, for example \cite{gross1}, \cite{roh80}, \cite{montgomery-rohrlich}, \cite{bugross85}. One finds in these references many examples which satisfy the assumption $L(E/H,1)\neq 0$ of Theorem \ref{main} (see also Corollary \ref{grosscurve}).
\begin{remark} \label{periodremark}For elements $\omega\in H^0(\ce,\Omega_{\ce/\co_F})$ and $\gamma_v\in H_1(E(F_v),\bz)$ we may define periods 
$$ \Omega_v:=\Omega(\gamma_v,\omega):=\int_{\gamma_v}\omega$$ 
in terms of which $\Omega$ and $\Omega(E)$ can be expressed as follows. First, there are
fractional $\co_K$-ideals $\fa(\gamma_v)$ and $\fa(\omega)$ such that
$$ H_1(E(F_v),\bz)= \fa(\gamma_v)\cdot\gamma_v$$
and 
$$ \mydet_{\co_K}H^0(\ce,\Omega_{\ce/\co_F})=\fa(\omega)\cdot\mydet_{\co_K}(\co_F\cdot\omega).$$
The trace map $\co_F\xrightarrow{\mathrm{Tr}}\co_K$ induces an isomorphism
$$\Hom_{\co_F}(H^0(\ce,\Omega_{\ce/\co_F}) ,\co_F)
\cong\Hom_{\co_K}(H^0(\ce,\Omega_{\ce/\co_F})\otimes_{\co_F}\cd^{-1}_{F/K},\co_K)$$
where $\cd^{-1}_{F/K}$ is the inverse different, an invertible $\co_F$-module. Assume for simplicity that $\cd^{-1}_{F/K}$ is a free $\co_K$-module and let $\beta_1,\dots,\beta_d$ be a basis. 
Then if we define
$$ \Omega:=\mydet\left(\int_{\gamma_v}\omega\otimes\beta_k\right)_{v,k}$$
we have
$$\fa(\Omega)=\fa(\omega)\cdot\prod_{v\mid\infty}\fa(\gamma_v).$$
Let $e_v$ be the indecomposable idempotents in
$$ \mathcal D^{-1}_{F/K}\otimes_{\co_K}K_\br\cong\co_F\otimes_{\co_K}K_\br\cong\prod_{v\mid\infty}F_v$$
and express the $\beta_k$ as a linear combination of the $e_v$. Then the base change matrix has determinant 
$$\mydet(\beta_{k,v})_{v,k}=\left(\sqrt{D_{F/K}}\right)^{-1}$$
where $D_{F/K}\in\co_K$ generates the relative discriminant ideal of the extension $F/K$ (and depends on the choice of $\beta_k$ by a factor in $(\co_K^\times)^2$ so that the $\co_K$-ideal generated by $\sqrt{D_{F/K}}$ is well defined). So we find
$$ \Omega=\left(\sqrt{D_{F/K}}\right)^{-1}\cdot \mydet\left(\int_{\gamma_v}\omega\otimes e_{v'}\right)_{v,v'}= 
 \left(\sqrt{D_{F/K}}\right)^{-1}\cdot\prod_{v\mid\infty}\Omega_v.$$
Denoting by $D_{L/\bq}\in\bz$ the discriminant of a number field $L$ we have
$$ |D_{F/\bq}|=N_{K/\bq}D_{F/K}\cdot |D_{K/\bq}|^{[F:K]}=D_{F/K}\overline{D_{F/K}}\cdot |D_{K/\bq}|^{[F:K]}$$
and therefore
\begin{align*}\bz\cdot\Omega(E)= &\fa(\omega)\overline{\fa(\omega)}\cdot\left(\sqrt{N_{K/\bq}D_{F/K}}\right)^{-1}\cdot \prod_{v\mid \infty}\Omega_v\overline{\Omega}_v\cdot \fa(\gamma_v)\overline{\fa(\gamma_v)}\\
= &\fa(\omega)\overline{\fa(\omega)}\cdot\left(\sqrt{|D_{F/\bq}|}\right)^{-1}\cdot \prod_{v\mid \infty}\sqrt{|D_{K/\bq}|}\cdot\Omega_v\overline{\Omega}_v\cdot \fa(\gamma_v)\overline{\fa(\gamma_v)}\\
= &\frac{1}{I_\omega}\cdot\left(\sqrt{|D_{F/\bq}|}\right)^{-1}\cdot \prod_{v\mid \infty}\mathrm{vol}_\omega(E(F_v))
\end{align*}
where 
$$ I_\omega:=(\fa(\omega)\overline{\fa(\omega)})^{-1}=[H^0(\ce,\Omega_{\ce/\co_F}):\co_F\cdot\omega]\in\bz$$
and the Haar measure on $E(F_v)$ is induced by the volume form $2dx\wedge dy=idz\wedge d\bar{z}$ after identifying the cotangent space of $E/F_v$ with  $F_v\simeq\bc$ via the basis $\omega$. This last form of the period term $\Omega(E)$ in the conjecture of Birch and Swinnerton-Dyer for abelian varieties over number fields can be found, for example, in \cite{conrad15}[L3, Ex. 6.4] (see also \cite{Flach-Siebel}[Lemma 18]) and continues to hold without our assumption of the existence of a basis $\beta_i$. However, in the above computation there will then be yet another corrective fractional $\co_K$-ideal $\fa(\beta_i)$ contributing to $\fa(\Omega)$.
\end{remark}

\begin{corollary}  (BSD for $E/F^+$) Under the assumptions of Theorem \ref{main} assume in addition that $E$ is defined over a subfield $F^+\subset F$ which is the fixed field of an involution of $F$ inducing complex conjugation on $K$. Then the groups $E(F^+)$ and $\Sha(E/F^+)$ are finite and
$$ \frac{L(E/F^+,1)}{\Omega(E^+)}=\frac{|\Sha(E/F^+)|}{|E(F^+)|^2}\cdot \prod_v|\Phi^+_v|$$
where $\Phi_v^+$ is the component group of the N\'eron model of $E/F^+$ at the prime $v$.
\label{BSD2}\end{corollary}

\begin{proof} If $F/F^+$ is a quadratic extension of number fields and $E/F^+$ an elliptic curve then there is an isogeny of abelian surfaces over $F^+$
$$ A:=Res_{F^+}^F(E/F)\sim E\times E_\epsilon$$
where $E_\epsilon$ is the twist of $E$ by the quadratic character $\epsilon$ attached to $F/F^+$. In our case $E_\epsilon$ is isogenous to $E$ (see \cite{milne72}[Thm. 3]), hence an isogeny
$$ A \sim E\times E.$$
We have isomorphisms $\Sha(E/F)\simeq \Sha(A/F^+)$, $E(F)\simeq A(F^+)$ and the BSD formula for $E/F$ is equivalent to the BSD formula for $A/F^+$ \cite{milne72}[Thm. 1]. By isogeny invariance of BSD we deduce the BSD formula for $(E\times E)/F^+$ (as well as finiteness of $\Sha((E\times E)/F^+)$ and $(E\times E)(F^+)$). Since all terms in the BSD formula for $(E\times E)/F^+$ are simply the squares of the corresponding terms in the BSD formula for $E/F^+$ we deduce the BSD formula for $E/F^+$ by taking square roots (see also \cite{milne72}[Cor. to Thm. 3] for this entire argument).
\end{proof}
\noindent Any CM elliptic curve $E/\bq$ with $L(E/\bq,1)\neq 0$ satisfies the assumptions of Corollary \ref{BSD2}. In particular we obtain the following. 
\begin{corollary}
The Birch and Swinnerton-Dyer conjecture is true for congruent number elliptic curves $E^{(n)}:ny^{2}=x^{3}-x$ for a density one subset of positive square-free integers $n\equiv 1,2,3 \mod{8}$.
\end{corollary}
\begin{proof}
By \cite{bt} we have $L(E^{(n)}/\bq,1)\neq 0$ for a density one subset of positive square-free integers $n\equiv 1,2,3 \mod{8}$.
\end{proof}
\begin{remark} 
Let $E/\bq$ be a CM elliptic curve and $\{E^{(n)}\}_{n}$ the  family of its quadratic twists over $\bq$. Then Corollary \ref{BSD2} in combination with \cite{rs15}[Thm. 3] implies that the distribution of orders of Tate-Shafarevich groups $\{\Sha(E^{(n)}/\bq)\}_{n}$ in the quadratic twist subfamily with analytic rank zero is as in \cite{rs15}[Thm. 3]. 
\end{remark}
\noindent We learnt the following application from B. Gross. 
\begin{corollary}
Let $p\equiv 7 \mod{8}$ be a prime, $K=\bq(\sqrt{-p})$ and $h$ the class number of $K$. Let $F=K(j)$ denote the Hilbert class field of $K$ for $j=j((1+\sqrt{-p})/2)$ and $F^{+}=\bq(j)$. Let $A(p)/F^{+}$ be the elliptic curve with CM by $\co_K$ over $F$ with Weierstrass equation 
$$
y^{2}=x^{3}+\frac{mp}{2^{3}3}x-\frac{np^{2}}{2^{5}3^{3}}
$$
where $m$ and $n$ are unique real numbers such that 
$$m^{3}=j,\\ -n^{2}p=j-1728\\ \text{ and }{\rm sgn}(n)=(\frac{2}{p}).$$
Then 
$$
|\Sha(A(p)/F^{+})|=
\bigg{(}\frac{1}{2^{h-1}}\cdot \frac{\prod_{\varphi\in\hat{{\rm Cl}}_{K}}\sum_{C \in {\rm Cl}_{K}}\varphi(C)t(C)}{\prod_{C\in {\rm Cl}_{K}}t(C)}\bigg{)}^{2}
$$
where $\hat{{\rm Cl}}_{K}$ denotes the character group of ${\rm Cl}_K$ and $t$ the modular function as in \cite{rv}[p.~562].
\label{grosscurve}\end{corollary}
\begin{proof}
By \cite{roh80} we have $L(A(p)/F^{+},1)\neq 0$. So the assertion follows from Corollary \ref{BSD2} and \cite{rv}[Thm. 8.2].
\end{proof}
\begin{corollary} Let $E/F'$ be an elliptic curve as in either Corollary \ref{BSD1} or Corollary \ref{BSD2} so that $F'$ is either $F$ or $F^+$. Let $X/F'$ be a principal homogeneous space of $E/F'$ and $\X\to\Spec(\co_{F'})$ a proper regular model of $X$. Then $\Br(\X)$ is finite and the special value conjecture  \cite{Flach-Morin-16}[Conj. 5.12] for $\zeta(\X,s)$ at $s=1$ holds true. More precisely, if the Zeta function $\zeta(\X,s)$ is factored as in \cite{Flach-Siebel}[Eq. (4)] 
$$ \zeta(\X,s)= \frac{\zeta_{F'}(s)\zeta_{F'}(s-1)}{\zeta(H^1,s)}$$
then
\begin{equation} \ord_{s=1}\zeta(H^1,s)=\rank_\bz\Pic^0(\X)\notag\end{equation}
and
\begin{equation} \zeta^*(H^1,1)=\frac{\#\Br(\overline{\X})\cdot\delta^2\cdot\Omega(\X)\cdot R(\X)}{(\#(\Pic^0(\X)_{tor}/\Pic(\co_{F'})))^2}\cdot\prod\limits_{\text{$v$ real}}\frac{\#\Phi_v}{\delta^2_v}\notag\end{equation}
where $\Pic^0(\X)$ is the kernel of the degree map on $\Pic(\X)$, $R(\X)$ is the regulator of the Arakelov intersection pairing on $\Pic^0(\X)$, $\Omega(\X)$ is the determinant of the period isomorphism between the finitely generated abelian groups $H^1(\X(\bc),2\pi i\cdot\bz)^{G_\br}$ and $H^1(\X,\co_\X)$ and
$$\Br(\overline{\X})=\ker\left(\Br(\X)\to\bigoplus_{\text{$v$ real}}\Br(X_{F'_v})\right).$$   The integer $\delta$ is the index of $X$, i.e. the g.c.d. of the degrees of all closed points, $\Phi_v=E(F'_v)/E(F'_v)^0$ is the group of components, and $\delta_v$ is the index of $X_{F'_v}$ over $F'_v$. 
\label{zeta}\end{corollary}

\begin{proof} By \cite{Flach-Siebel}[Thm. 6.1] the BSD formula for $E/F'$ is equivalent to the special value conjecture \cite{Flach-Siebel}[Eq. (6)] for $\zeta^*(H^1,1)$. Since $E$ has genus $1$ the equality $\delta_v'=\delta_v$ of period and index for real $v$ in \cite{Flach-Siebel}[Eq. (6)] follows from the proof of \cite{Flach-Siebel}[Lemma 9].
\end{proof}

\begin{remark} In the situation of Corollary \ref{BSD2} the extension $F/\bq$ is Galois since $\Aut(F)$ contains the involution $\sigma$ in addition to $\Gal(F/K)$ so that $\#\Aut(F)=[F:\bq]$. The number $r_1$ of real places of $F^+$ is the number of fixed points of the action of $\sigma$ on the set of archimedean places of $F$. If one chooses $\sigma\in\Aut(F)$ as the restriction of complex conjugation with respect to a particular complex embedding there is at least one fixed point, and the total number of fixed points of $\sigma$ coincides with the number of fixed points of the conjugation action of $\Gal(K/\bq)$ on $\Gal(F/K)$. Hence the signature $(r_1,r_2)$ of the field $F^+$ in Corollary \ref{BSD2} either satisfies $r_1=0$ or $r_1\mid 2r_2$. One can construct examples of fields $F^+$ for any $(r_1,r_2)$ with $r_1\mid 2r_2$: choose $$\Gal(F/K)\simeq \bz/r'_1\bz\times \bz/((r_1+2r_2)/r'_1)\bz$$ with $\Gal(K/\bq)$ acting trivially on the first factor and by $-1$ on the second. Here $r_1'=r_1$ if $r_1$ is odd and $r_1'=r_1/2$ if $r_1$ is even. To construct examples of Corollaries \ref{BSD2} and \ref{zeta} one may take $E/F$ with $j(E)\in\bq$ but one also needs non-vanishing results for $L(\overline{\psi},1)$. 
\label{r1}\end{remark}

\begin{theorem} Let $A/K$ be a CM abelian variety with $\End_K(A)\simeq \co_L$ for some CM field $L$ with $[L:\bq]=2\dim(A)$. Let $\varphi$ be its Serre-Tate character and assume
$$L(\bar{\varphi},1)\neq 0.$$ Let $\Omega\in L_\br^\times$ and $\fa(\Omega)$ be the period and fractional $\co_L$-ideal defined in Def. \ref{perioddef} in section \ref{preliminary}. 
Then $\Sha(A/K)$ and $A(K)$ are finite, $\frac{L(\bar{\varphi},1)}{\Omega}\in L^\times$ and
$$ \frac{L(\bar{\varphi},1)}{\Omega}=\frac{|\Sha(A/K)|_{L}}{|A(K)|_{L} |{^t}\! A(K)|_{L}}\cdot \prod_v|\Phi_v|_{L}\cdot\fa(\Omega)$$
in the group of fractional ideals of $L$. 
\label{main2}\end{theorem}

Apart from being a special case of the equivariant Tamagawa number conjecture \cite{bufl01} this $L$-equivariant Birch and Swinnerton-Dyer conjecture was also formulated by Buhler and Gross \cite{bugross85}[Conj. 12.3] in the special case where $A=\Res_{H/K}E$ for $E/H$ a CM elliptic curve over the Hilbert class field and $[H:K]$ odd. We will explicate the connection to \cite{bugross85}[Conj. 12.3] in Prop. \ref{bugr} in section \ref{sec:ab}.

\begin{corollary} (BSD for $A/K$) Under the assumptions of Thm. \ref{main2} we have
$$ \frac{L(A/K,1)}{\Omega(A)}=\frac{|\Sha(A/K)|}{|A(K)|\cdot  |{^t}\! A(K)|}\cdot \prod_v|\Phi_v| $$
where the period $\Omega(A)$ is defined for example in \cite{Flach-Siebel}[Lemma 18].  
\label{abBSD}\end{corollary}

\begin{proof} This follows by taking the norm from $L$ to $\bq$ of the identity in Thm. \ref{main2}.
\end{proof}

\begin{corollary} (BSD for $A/\bq$) In the situation of Thm. \ref{main2} assume in addition that $A$ is defined over $\bq$. Then we have
$$ \frac{L(A/\bq,1)}{\Omega(A^+)}=\frac{|\Sha(A/\bq)|}{|A(\bq)|\cdot  |{^t}\! A(\bq)|}\cdot \prod_v|\Phi^+_v| $$
where the period $\Omega(A^+)$ is the period of \cite{Flach-Siebel}[Lemma 18] for $A/\bq$.  
\label{abelianBSD2}\end{corollary}

\begin{proof} This follows as in the proof of Cor. \ref{BSD2}.
\end{proof}

\begin{corollary} Let $f$ be an elliptic newform of weight 2, level $N$ and arbitrary character, and let $A_f$ be the isogeny factor of the Jacobian of $X_1(N)$ associated to $f$ by Eichler-Shimura. If $f$ has CM and $L(A_f,1)\neq 0$ then the Birch and Swinnerton-Dyer conjecture holds for $A_f$.
\end{corollary} 

\begin{proof} Let $K$ be the CM field of $f$ and $L_0:=\End_\bq(A_f)_\bq$ the  
field generated by the Hecke eigenvalues of $f$ \cite{ribet80}[Cor. 4.2]. 
Then the base change $A_{f,K}$ of $A_f$ to $K$ is 
either simple with $\End_K(A_{f,K})_\bq\simeq L:=L_0K$
or $A_{f,K}\sim A_1\times A_2$ with $\End_K(A_i)_\bq\simeq L:=L_0$. 
 By isogeny invariance of BSD we can assume that $A_{f,K}$ has multiplications by the maximal order 
in either case. 
 Then BSD holds for $A_{f,K}$ by Cor. \ref{abBSD} and follows for $A_f$ as in Cor. \ref{BSD2}.
\end{proof}

{\em Remarks on the proof and related work.} The proof of Theorem \ref{main} is naturally situated in the framework of the equivariant Tamagawa number conjecture.  
Let $\fp$ be a prime ideal of $\co_K$ and $p$ the rational prime below. 
The proof begins with a reduction of 
Theorem \ref{main} to the existence of an equivariant zeta element for $E/F$, i.e. a basis $z$ of the $\co_K$-equivariant determinant of the $p$-adic \'etale cohomology of $E/F$ which also encodes the L-value $L(E/F,1)$ (Prop. \ref{keyelliptic}). 
By a descent of the Iwasawa main conjecture for $K$ \cite{jlk}, 
such a basis $z$ is constructed via elliptic units (subsections \ref{sec:descent} and \ref{sec:ell}) whose link to $L(E/F,1)$ is given by an explicit reciprocity law (subsection \ref{sec:erl}).
The main conjecture \cite{jlk} expresses the determinant of Iwasawa cohomology of $K$ in terms of elliptic units,   
as pioneered by Kato \cite{kato92}, \cite{kato93}. 
Our descent of the main conjecture is formulated in terms of perfect complexes and is uniform for any prime $\fp$. 
It first leads to the finiteness of $E(F)$ and $\Sha(E/F)_{\fp^\infty}$ (offering another proof of results of \cite{coates77}, \cite{rubin87}) and then to $z$.

The above approach to the $\fp$-primary part of Gross' conjecture and the BSD formula differs from that of Rubin \cite{rubin91}. The calculus of determinants of perfect complexes does not appear in \cite{rubin91}. The descent \cite{rubin91}[\S11] involves classical Iwasawa modules and is more involved for primes $p$ non-split in $K$ \cite{rubin91}[pp. 61-66].

The CM elliptic curve $X_{0}(49)$ has analytic rank $0$ and it was the first elliptic curve for which the full BSD conjecture was proved \cite[\S22]{gross1}, \cite{rubin87}. 
(See also the discussion in \cite{gross3}[p. 17].)
Since Rubin's fundamental work \cite{rubin91}, the $\fp$-primary part of the  
BSD formula for CM elliptic curves with analytic rank $0$ and primes $\fp$ such that
$$
\fp \big{|} |\co_K^\times|\cdot[F:K]\cdot\mathrm{disc}(F/K)
$$
has been much studied, 
especially the case of CM elliptic curves over $\bq$ and the prime $p=2$. 
This includes the extensive work of Coates \cite{coates91-2}, \cite{coates13}, \cite{coates14}, \cite{tian15}, \cite{coates18}, \cite{coates21}, \cite{coates22}, Kezuka \cite{kezuka18}, \cite{kezuka19}, \cite{kezuka20}, \cite{kezuka21}, Tian 
\cite{tian00}, \cite{tian12}, \cite{tian14}, \cite{tian14-2}, \cite{tian17a}, \cite{tian22}, Tian-Yuan-Zhang \cite{tian17b}, Zhao \cite{zhao97}, \cite{zhao01}, \cite{zhao03}, \cite{zhao05}, as well as 
\cite{rubin94}, \cite{gonzalez97}, \cite{qz}, \cite{ccl}, \cite{choi}, \cite{hsy}, \cite{zhai}, \cite{rosu}, \cite{clz}, \cite{sz}. 
The previous work was the original impetus for our study.    
Coates-Kezuka-Li-Tian \cite{coates22} show the $2$-part of the BSD formula for CM elliptic curves with ordinary reduction at $2$. 
Some of the previous studies focus on specific families, for instance Tian \cite{tian14} and Tian-Yuan-Zhang \cite{tian17b} show key results for congruent number elliptic curves (which also led to Smith's work \cite{sm1}, \cite{sm2}). 
The proofs employ various tools such as the Euler system of elliptic units, explicit Waldspurger formula and congruences between modular forms. 
Yet, prior to Corollary \ref{BSD2}, the $p$-part of the BSD formula for CM elliptic curves  
over $\bq$ with analytic rank $0$ and $p||\co_K^\times|$ remained open in general. 

Our approach to the BSD formula seems amenable to other situations. In future work we plan to consider  
the case of CM elliptic curves with analytic rank $1$.  
\subsection*{Acknowledgements} We would like to thank Benedict Gross, Shinichi Kobayashi, Karl Rubin, Chris Skinner, Ye Tian and Shou-Wu Zhang for helpful comments and discussions. 
We are grateful to the referee for valuable suggestions. 

We were looking forward to send a copy of this article to John Coates for comments when we received the sudden news of his passing. In a sad and strange coincidence of events it is now entirely fitting to dedicate this article to his memory. John not only initiated this particular field of study some 45 years ago \cite{coates20} but he also kept constantly moving it forward over the years with his characteristic passion and energy, being an inspiration to many generations of mathematicians.
 
\section{Preliminary reductions}\label{preliminary}

In this section we reduce the proof of Thm. \ref{main}, resp. Thm. \ref{main2}, to the existence of a basis of the determinant of global Galois cohomology of the Tate module with certain properties, see Prop. \ref{keyelliptic}, resp. Prop. \ref{keyabelian}. Such a basis will then be provided by the combination of Kato's reciprocity law with the Iwasawa main conjecture in the next section.  We present the initial reduction step in the slightly more general context of an abelian variety which is not necessarily CM. This initial reduction step is in principle well known \cite{kato93}[Ch.I. 2.3] and has an analogue for any motive over a number field (see for example \cite{bufl96}[p. 85/86]).

Let $A/F$ be an abelian variety over a number field $F$ with dual abelian variety ${^t}\! A/F$ and denote by $\ca/\co_F$, resp. ${^t}\!\ca/\co_F$, the N\'eron model of $A$, resp. ${^t}\! A$. Let $L$ be a number field so that there is given an embedding  
$$\co_L\to\End_F(A).$$
This induces an embedding $\co_L\to\End_F({^t}\!A)$ by functoriality. We fix a prime number $p$ and define 
$$ L_p:=L\otimes_\bq\bq_p\simeq\prod_{\fp\mid p}L_\fp;\quad \co_{L_p}:=\co_L\otimes_\bz\bz_p\simeq\prod_{\fp\mid p}\co_{L_\fp}$$
and
$$ T:=T_p({^t}\!A)\simeq H^1(A_{\bar{F}},\bz_p(1));\quad V:=V_p({^t}\!A):= T_p({^t}\!A)\otimes_{\bz_p}\bq_p$$
where $T_p({^t}\!A)$ is the ${p}$-adic Tate module of ${^t}\!A$. Let $S$ be a finite set of places of $F$ containing all archimedean places, all places above $p$ and all places of bad reduction. Then we may view $T$ as a smooth sheaf of $\co_{L_p}$-modules on $\Spec(\co_{F,S})_{et}$ and we denote by $R\Gamma(\co_{F,S},T)$ its \'etale cohomology. For each prime $v\mid p$ of $F$ let
$$ H^1(F_v,V) \xrightarrow{\exp_v^*} D^0_{dR}(V)\simeq H^0(A_{F_v},\Omega_{ A_{F_v}/F_v})$$
be the dual exponential map of the $\Gal(\bar{F}_v/F_v)$ -representation $V$ \cite{kato93}[Ch. II, Thm. 1.4.1].
For any prime $v$ of $F$ denote by 
$$ P_v({^t}\!A/F,t)=\mydet_{L_l}(1-\Frob_v^{-1}\cdot t\vert H^1({^t}\!A_{\bar{F}},\bq_l)^{I_v})\in\co_{L}[t]$$
the Euler factor of the $L$-equivariant Hasse-Weil L-function of ${^t}\!A/F$ (which is independent of the auxiliary prime $v\nmid l$) and let $\Phi_v$ be the component group of $\ca$ at $v$. Denote by $|M|_{L_p}$ the part of the order ideal of a finite $\co_{L}$-module $M$ supported in $\{\fp\mid p\}$.

\begin{prop} With the notation just introduced the following hold.

a) If $A(F)$ and $\Sha(A/F)_{p^\infty}$ are finite then the composite map
$$ H^1(\co_{F,S},V)\xrightarrow{}\prod_{v\mid p} H^1(F_v,V) \xrightarrow{\prod_{v\mid p}\exp_v^*} \prod_{v\mid p} H^0( A_{F_v},\Omega_{ A_{F_v}/F_v})$$
is an isomorphism, and $H^i(\co_{F,S},V)=0$ for $i\neq 1$.  We obtain an induced isomorphism 
$$ \iota: \mydet_{L_p}^{-1} R\Gamma(\co_{F,S},V) \simeq  \mydet_{L_p}H^1(\co_{F,S},V) \simeq  \mydet_{L_p} H^0( A,\Omega_{A/F})\otimes_{\bq}\bq_p.$$

b) Assume in addition that $F$ has no real embedding or that $p>2$. Then $R\Gamma(\co_{F,S},T)$ is a perfect complex of $\co_{L_p}$-modules and 
$$  \iota\left(\mydet_{\co_{L_p}}^{-1} R\Gamma(\co_{F,S},T)\right)=\frac{|\Sha(A/F)|_{L_p}}{|A(F)|_{L_p} |{^t}\! A(F)|_{L_p}}\cdot \prod_v|\Phi_v|_{L_p}\cdot \prod_{v\in S} P_v({^t}\!A/F,Nv^{-1})\cdot \Upsilon_p$$
where $\Upsilon_p:=\Upsilon\otimes_{\co_L}\co_{L_p}$,
$$ \Upsilon :=  \mydet_{\co_{L}} \left(H^0(\ca,\Omega_{ \ca/\co_F})\otimes_{\co_F}\cd_{F/\bq}^{-1}\right)$$
and $\cd_{F/\bq}^{-1}$ is the inverse different of the extension $F/\bq$.
\label{key}\end{prop}

\begin{proof} Consider the following diagram of complexes of $\co_{L_p}$-modules with exact rows and columns
\begin{equation}\begin{CD} R\Gamma_c(\co_{F,S},T) @>>> R\Gamma_f(F,T) @>>> \bigoplus\limits_{v\in S} R\Gamma_f(F_v,T)\\
\Vert @. @VVV@VV\oplus_v \alpha_v V\\
R\Gamma_c(\co_{F,S},T) @>>> R\Gamma(\co_{F,S},T) @>>> \bigoplus\limits_{v\in S} R\Gamma(F_v,T)\\
@. @VV\beta V @VVV\\
{} @. \bigoplus\limits_{v\in S}R\Gamma_{/f}(F_v,T) @= \bigoplus\limits_{v\in S}R\Gamma_{/f}(F_v,T).
\end{CD}\label{wichtig}\end{equation}
Here, following \cite{bk88} we define
$$ R\Gamma_f(F_v,T) = \begin{cases} {^t}\!A(F_v)^\wedge[-1] & v\nmid\infty\\
\tau^{\leq 1} R\Gamma(F_v,T) & v\mid\infty\end{cases}$$
where for any abelian group $M$ we denote by 
\begin{equation} M^\wedge:= \varprojlim_\nu M/p^\nu \label{completion}\end{equation}
its (underived) $p$-adic completion. Since $H^0(F_v,T)=0$ for $v\nmid\infty$ there is a map of complexes
$$\alpha_v: {^t}\!A(F_v)^\wedge[-1]\xrightarrow{\tilde{\alpha}_v} H^1(F_v,T)[-1]\to R\Gamma(F_v,T)$$
where $\tilde{\alpha}_v$ is the inverse limit of the connecting homomorphisms $${^t}\!A(F_v)/p^\nu\to H^1(F_v,{^t}\!A_{p^\nu})$$ induced by the Kummer sequence. The complex $R\Gamma_{/f}(F_v,T)$ is defined as the mapping cone of $\alpha_v$ and the complex $R\Gamma_f(F,T)$ as the mapping fibre of $\beta$.

\begin{lemma} We have
$$ H^i_f(F,T) \simeq\begin{cases} {^t}\!A(F)^\wedge & i=1 \\ \Sha({^t}\!A/F)^\wedge \oplus \Hom_{\bz_p}(A(F)^\wedge,\bz_p) & i=2\\
 \Hom_{\bz_p}(A(F)^\wedge_{tor},\bq_p/\bz_p) & i=3\\ 0 & i\neq 1,2,3\end{cases} $$  
\label{rgf}\end{lemma}

\begin{proof} First note that $R\Gamma_{/f}(F_v,T)$ is concentrated in degrees $1,2$ for $v\nmid\infty$ and there is an isomorphism 
$$ H^i(\co_{F,S},T)\simeq \bigoplus_{\text{$v$ real}} H^i(F_v,T)\simeq \bigoplus_{\text{$v$ real}} H^i_{/f}(F_v,T)$$
for $i\geq 3$ by \cite{milduality}[Prop. II.2.9, Thm. I.4.10]. It follows that $R\Gamma_f(F,T)$ is concentrated in degrees $0\leq i\leq 3$. 
The long exact sequence associated to the middle column in (\ref{wichtig}) gives 
$$ H^0_f(F,T)=H^0(\co_{F,S},T)=0$$
and 
$$ H^1_f(F,T)=\ker\left(H^1(\co_{F,S},T)\to \bigoplus\limits_{v\in S} \frac{H^1(F_v,T)}{{^t}\!A(F_v)^\wedge}\right).$$
Recall that the classical Selmer group $\mathrm{Sel}(F,{^t}\!A_{p^\nu})$ can be defined as
$$ \mathrm{Sel}(F,{^t}\!A_{p^\nu})= \ker\left(H^1(\co_{F,S},{^t}\!A_{p^\nu})\to \bigoplus\limits_{v\in S} \frac{H^1(F_v,{^t}\!A_{p^\nu})}{{^t}\!A(F_v)/p^\nu}\right)$$
since the image of ${^t}\!A(F_v)/p^\nu$ in $H^1(F_v,{^t}\!A_{p^\nu})$ coincides with the unramified classes for $v\notin S$.
Taking the inverse limit over $\nu$ in the short exact sequence
$$ 0\to {^t}\!A(F)/p^\nu\to \mathrm{Sel}(F,{^t}\!A_{p^\nu})\to \Sha({^t}\!A/F)_{p^\nu}\to 0$$
and using finiteness of $\Sha({^t}\!A/F)_{p^\infty}$
we find 
$${^t}\!A(F)^\wedge\simeq \varprojlim_\nu \mathrm{Sel}(F,{^t}\!A_{p^\nu})\simeq H^1_f(F,T).$$
The long exact sequence associated to the top row in (\ref{wichtig}) gives an exact sequence
$$ \bigoplus_{v\in S} H^1_f(F_v,T) \to H^2_c(\co_{F,S},T)\to H^2_f(F,T)\to 0$$
and an isomorphism
$$ H^3_c(\co_{F,S},T)\simeq H^3_f(F,T).$$
Using Artin-Verdier duality \cite{milduality}[Cor. II.3.3] and the fact that our $R\Gamma_c$ agrees with that of loc. cit. (formed with Tate cohomology at the infinite places) in degrees $\geq 2$ we find an exact sequence
$$ 0\to H^2_f(F,T)^*\to H^1(\co_{F,S}, A_{p^\infty})\to  \bigoplus\limits_{v\in S} \frac{H^1(F_v,A_{p^\infty})}{A(F_v)\otimes\bq_p/\bz_p}$$
and an isomorphism
$$ H^3_f(F,T)^* \simeq H^0(\co_{F,S}, A_{p^\infty}).$$
Here we use the definition
$$ M^*:=\Hom_{\bz_p}(M,\bq_p/\bz_p),$$
the isomorphism of $\pi_1(\Spec(\co_{F,S}))$-modules
\begin{equation} T^*(1)\simeq A_{p^\infty}\label{dualmodule}\end{equation}
and the fact that the orthogonal complement of ${^t}\!A(F_v)^\wedge$ under the perfect pairing
$$ H^1(F_v,T)\times H^1(F_v, A_{p^\infty})\to H^2(F_v,\bq_p/\bz_p(1))\simeq\bq_p/\bz_p$$
is $A(F_v)\otimes\bq_p/\bz_p$ \cite{milduality}[Cor. I.3.4, Rem. I.3.5]. Hence we obtain an isomorphism $$H^2_f(F,T)^*\simeq\mathrm{Sel}(F,A_{p^\infty}).$$ Dualizing again we find an exact sequence
$$ 0\to \Sha(A/F)_{p^\infty}^*\to H^2_f(F,T)\to ( A(F)^\wedge\otimes_{\bz_p}\bq_p/\bz_p)^*\to 0$$
and an isomorphism
$$ H^3_f(F,T)\simeq (A(F)^\wedge_{tor})^*.$$
By \cite{milduality}[Thm. I.6.13] if $\Sha(A/F)_{p^\infty}$ is finite there is a non-degenerate pairing
$$ \Sha({^t}\!A/F)^\wedge\times \Sha(A/F)_{p^\infty}\to\bq_p/\bz_p$$
and for any $\bz_p$-module $M$ there is an isomorphism
\begin{align}&\Hom_{\bz_p}(M\otimes_{\bz_p}\bq_p/\bz_p,\bq_p/\bz_p)\label{zpdual}\\
\simeq &\Hom_{\bz_p}(M, \Hom_{\bz_p}(\bq_p/\bz_p,\bq_p/\bz_p))\notag\\
\simeq &\Hom_{\bz_p}(M,\bz_p).\notag\end{align}
Hence we find an exact sequence
$$0\to\Sha({^t}\!A/F)^\wedge\to H^2_f(F,T)\to \Hom_{\bz_p}(A(F)^\wedge,\bz_p)\to 0$$
concluding the proof of Lemma \ref{rgf}.
\end{proof}

\begin{lemma}  We have
$$ H^i_{/f}(F_v,T) =\begin{cases} \Hom_{\bz_p}(A(F_v)^\wedge,\bz_p) & i=1, v\mid p\\
 \Hom_{\bz_p}(A(F_v)^\wedge_{tor},\bq_p/\bz_p) & i=2, v\nmid \infty\\ 
  H^i(F_v,T) & i\geq 3, v\mid\infty\\ 0 & \text{else.}\end{cases} $$  
In particular, for $v\nmid\infty$ there is a quasi-isomorphism of complexes of $\co_{L_p}$-modules
\begin{equation}R\Gamma_{/f}(F_v,T)[1] \simeq R\Hom_{\bz_p}(A(F_v)^\wedge,\bz_p).\label{rgmodfderived}\end{equation}
\label{rgmodf}\end{lemma}

\begin{proof} The Kummer sequence 
$$ 0\to {^t}\! A(F_v)^\wedge\to H^1(F_v,T)\to \varprojlim_\nu H^1(F_v,{^t}\! A)_{p^\nu}\to 0$$
together with duality for abelian varieties over local fields \cite{milduality}[Cor. I.3.4] and (\ref{zpdual})
$$ \varprojlim_\nu H^1(F_v,{^t}\! A)_{p^\nu}\simeq \left(\varinjlim_\nu\, A(F_v)/p^\nu\right)^*\simeq \Hom_{\bz_p}(A(F_v)^\wedge,\bz_p)$$
give the Lemma for $i=1$. Note that these groups vanish unless $v\mid p$. The statement for $i=2$ follows from (\ref{dualmodule}) and Tate local duality \cite{milduality}[Cor. I.2.3]
$$ H^2_{/f}(F_v,T)\simeq H^2(F_v,T)\simeq  H^0(F_v, A_{p^\infty})^*\simeq\left( A(F_v)^\wedge_{tor}\right)^*.$$
The statement for $i\geq 3$ is just the definition of $R\Gamma_{/f}(F_v,T)$. Note here that $H^i(F_v,T)=0$ for $i=2$ and $v\mid\infty$ and for $i\geq 3$ and $v\nmid \infty$. The isomorphism (\ref{rgmodfderived}) can be proved either via a version of Tate duality in the derived category, or by direct inspection since $ \Hom_{\bz_p}(A(F_v)^\wedge_{tor},\bq_p/\bz_p)\simeq
\Ext^1_{\bz_p}( A(F_v)^\wedge,\bz_p)$, and since any bounded complex of $\co_{L_p}$-modules is quasi-isomorphic to the sum of its cohomology groups (placed in their respective degrees).

\end{proof}

For $v\mid p$ the dual exponential map
\begin{equation} H^1_{/f}(F_v, V_p({^t}\! A))\xrightarrow{\exp_v^*} H^0(A_{F_v},\Omega_{A_{F_v}/F_v})\label{dualexp}\end{equation}
is an isomorphism since its dual \cite{kato93}[Ch. II, Thm. 1.4.1]
$$ \mathrm{Lie}(A_{F_v})\xrightarrow{\exp_v}H^1_f(F_v, V_p(A))\simeq A(F_v)^\wedge\otimes_{\bz_p}\bq_p$$
is an isomorphism. This is because the inverse $\log_v$ of $\exp_v$ (the formal group logarithm) induces an isomorphism
\begin{equation}  \log_v:\hat{\ca}(\fm_v^n)\xrightarrow{\sim}\fm_v^n\mathrm{Lie}(\ca_{\co_{F_v}}) \label{formalgrouplog}\end{equation}
for large enough $n$ and
$$ (\fm_v^n\mathrm{Lie}(\ca_{\co_{F_v}}))\otimes_{\bz_p}\bq_p= \mathrm{Lie}(A_{F_v});\quad \hat{\ca}(\fm_v^n)\otimes_{\bz_p}\bq_p=A(F_v)^\wedge\otimes_{\bz_p}\bq_p.$$
Here $\fm_v$ is the maximal ideal of $\co_{F_v}$ and $\hat{\ca}$ is the formal completion of $\ca_{\co_{F_v}}$ at the identity section. If now $A(F)$ and $\Sha(A/F)_{p^\infty}$ are both finite then so are ${^t}\! A(F)$ and $\Sha({^t}\! A/F)_{p^\infty}$ and it follows from Lemmas \ref{rgf} and \ref{rgmodf} that
$$ R\Gamma_f(F,V)\simeq 0;\quad R\Gamma_{/f}(F_v,V)\simeq\begin{cases}H^1_{/f}(F_v, V)[-1] & v\mid p\\ 0&\text{else.} \end{cases} $$
The middle vertical exact triangle in (\ref{wichtig}) then implies part a) of Prop. \ref{key}.

\begin{lemma} For $v\mid p$ let $\iota_v$ be the isomorphism
$$ \iota_v: \mydet_{L_p}^{-1} R\Gamma_{/f}(F_v,V) \simeq  \mydet_{L_p}H^1_{/f}(F_v,V) \simeq  \mydet_{L_p} 
H^0(A_{F_v},\Omega_{ A_{F_v}/F_v})$$
induced by the dual exponential map (\ref{dualexp}). For $v\nmid p\infty$ let
$$ \iota_v: \mydet_{L_p}^{-1} R\Gamma_{/f}(F_v,V) \simeq L_p$$
be the isomorphism arising from acyclicity of $R\Gamma_{/f}(F_v,V)$.
Then
\begin{equation}  \iota_v\left(\mydet_{\co_{L_p}}^{-1} R\Gamma_{/f}(F_v,T)\right)=|{^t}\Phi_v|_{L_p}\cdot  P_v({^t}\!A/F,Nv^{-1})\cdot \Upsilon_v\label{vol}\end{equation}
where
$$ \Upsilon_v := \begin{cases} \mydet_{\co_{L_p}} \left(H^0(\ca_{\co_{F_v}},\Omega_{ \ca_{\co_{F_v}}/\co_{F_v}})\otimes_{\co_{F_v}}\cd^{-1}_{F_v/\bq_p}\right) & v\mid p \\ \co_{L_p} & v\nmid p\end{cases}$$
and $\cd^{-1}_{F_v/\bq_p}$ is the inverse different of the extension $F_v/\bq_p$.
\label{localvolume} \end{lemma}

\begin{proof} Let $\ca^0/\co_F$ be the open sub-group scheme of $\ca/\co_F$ so that  $\ca^0_{\kappa_v}$ is the connected component of the identity of $\ca_{\kappa_v}$ for each residue field $\kappa_v$ of $\co_F$. We have a filtration of the group $A(F_v)=\ca(\co_{F_v})$ given by exact sequences of $\co_L$-modules
$$ 0\to \ca^0(\co_{F_v})\to A(F_v)\to \Phi_v\to 0$$
and
$$ 0\to \hat{\ca}(\fm_v)\to \ca^0(\co_{F_v})\to \ca^0(\kappa_v)\to 0.$$
Since all these groups have bounded $p$-primary torsion, the $p$-adic completion functor (\ref{completion}) is exact and we obtain
exact sequences of $\co_{L_p}$-modules
\begin{equation} 0\to \ca^0(\co_{F_v})^\wedge\to A(F_v)^\wedge\to \Phi_v^\wedge\to 0\label{componentseq}\end{equation}
and
\begin{equation} 0\to \hat{\ca}(\fm_v)^\wedge\to \ca^0(\co_{F_v})^\wedge\to \ca^0(\kappa_v)^\wedge\to 0.\label{formalseq}\end{equation}

\begin{lemma} For any finite place $v$ of $F$ and any prime $p$ there is an identity of fractional $\co_{L_p}$-ideals
\begin{equation} \frac{|\ca^0(\kappa_v)^\wedge|_{L_p}}{|\Lie(\ca_{\kappa_v})^\wedge|_{L_p}}=P_v({^t}\!A/F,Nv^{-1}).\label{euler2}\end{equation}
\end{lemma}

\begin{proof} 
The smooth, connected commutative group scheme $\ca^0_{\kappa_v}$ over the perfect field $\kappa_v$ has a filtration, preserved by any endomorphism,
\begin{equation} 0\to U\to \ca^0_{\kappa_v}\to B\to 0\label{uni}\end{equation}
where $U$ is unipotent (and smooth and connected) and $B$ is semiabelian (combine Chevalley's theorem \cite{conrad02} with \cite{sga3}[XVII Thm. 7.2.1]). We claim that
\begin{equation} \frac{|U(\kappa_v)^\wedge|_{L_p}}{|\Lie(U)^\wedge|_{L_p}}=1.\label{uvanishing}\end{equation}
The group scheme $U$ has a filtration with successive quotients $\bg_a$  \cite{sga3}[XVII, Prop. 4.1.1]. The Lie algebra functor being exact for smooth group schemes, there is a corresponding filtration of $\Lie(U)$. Since $H^1(G_{\kappa_v},U')=0$ for any connected group scheme $U'/\kappa_v$ there is also a corresponding filtration of $U(\kappa_v)$ with successive quotients $\bg_a(\kappa_v)=\kappa_v$. So for $v\nmid p$ we have $U(\kappa_v)^\wedge=\Lie(U)^\wedge=0$ and (\ref{uvanishing}) holds.
Since $p$ annihilates $\bg_a$ for $v\mid p$ some power $p^\nu$ annihilates $U$, and the action of $\co_L$ on $U$ factors through the finite semilocal ring $\co_L/p^\nu$.
The indecomposable idempotents $e_1,\dots, e_r$ of $\co_L/p^\nu$ act by algebraic endomorphisms on $U$, so have closed image $e_iU$ and 
$$ U\simeq e_1U\times\cdots\times e_rU.$$
To prove (\ref{uvanishing}) it suffices to show 
$$\mathrm{length}_{e_i\co_L/p^\nu}(e_iU(\kappa_v))=\mathrm{length}_{e_i\co_L/p^\nu}(\Lie(e_iU))$$
for $i=1,\dots,r$. This follows from the fact that $e_iU$ is itself unipotent, smooth, connected, hence has a filtration with subquotients $\bg_a$ and 
$\bg_a(\kappa_v)\simeq\kappa_v\simeq \Lie(\bg_a)$.

By similar reasoning the filtration (\ref{uni}) induces corresponding filtrations on $\Lie(\ca^0_{\kappa_v})\simeq \Lie(\ca_{\kappa_v})$ and on $\ca^0(\kappa_v)$. 
It therefore suffices to show 
\begin{equation}
\frac{|B(\kappa_v)^\wedge|_{L_p}}{|\Lie(B)^\wedge|_{L_p}}=P_v({^t}\!A/F,Nv^{-1}).\label{euler3}\end{equation}
For $v\nmid p$ we have
\begin{align} P_v({^t}\!A/F,Nv^{-1})=&\mydet_{L_p}(1-\Frob_v^{-1}\cdot Nv^{-1}\vert H^1({^t}\!A_{\bar{F}},\bq_p)^{I_v})\notag\\
=& \mydet_{L_p}(1-\Frob_v^{-1}\vert H^1({^t}\!A_{\bar{F}},\bq_p(1))^{I_v})\notag\\
=& \mydet_{L_p}(1-\Frob_v^{-1}\vert V_p(A)^{I_v})\notag\\
=& \mydet_{L_p}(1-\Frob_v^{-1}\vert V_p(B))
\notag\end{align}
where the last identity is \cite{sga7}[IX, Prop. 2.2.5]. Moreover $\Lie(B)^\wedge=0$ and $\Frob_v$ acts invertibly on $T_p(B)$. The exact sequence\footnote{It arises as follows. The snake lemma applied to 
\begin{equation}\begin{CD} 
0 @>>> B(\kappa_v) @>>> B(\overline{\kappa}_{v}) @>\Frob_v-1 >> B(\overline{\kappa}_{v}) @>>>0\\
@. @VVp^{n}V @VVp^{n}V@VVp^{n} V @.\\
0 @>>> B(\kappa_v) @>>> B(\overline{\kappa}_{v}) @>\Frob_v-1 >>B(\overline{\kappa}_{v})@>>>0\\ 
\end{CD}\label{frob}\end{equation}
gives an exact sequence of finite $\co_{L_{p}}$-modules
$$
0 \to B(\kappa_v)[p^n] \to B(\overline{\kappa}_{v})[p^{n}] \xrightarrow{\Frob_{v}-1} B(\overline{\kappa}_{v})[p^{n}] \to B(\kappa_{v})/p^n \to 0
$$
to which one applies the projective limit over $n$ (an exact functor in this case).} of $\co_{L_p}$-modules
$$ 0\to T_p(B)\xrightarrow{\Frob_v-1}T_p(B)\to B(\kappa_v)^\wedge\to 0$$
then shows that 
$$ |B(\kappa_v)^\wedge|_{L_p}=\mydet_{L_p}(\Frob_v-1\vert V_p(B))\sim_{\co_{L_p}^\times}\mydet_{L_p}(1-\Frob_v^{-1}\vert V_p(B))$$
verifying (\ref{euler3}). 

For $v\mid p$ let $D(B)$ be the covariant Dieudonn\'e module of the $p$-divisible group $B[p^\infty]$ associated to $B/\kappa_v$ \cite{chaioort}[Thm. 4.33]. This is a free $W(\kappa_v)$-module so that  
\begin{align*}&P_v({^t}\!A/F,Nv^{-1})\\= &\mydet_{L_p\otimes W(\kappa_v)}(1-\Frob_v^{-1}\vert D(B)_\bq)\\
= & \mydet_{L_p\otimes W(\kappa_v)}(\Frob_v-1\vert D(B)_\bq)\cdot  \mydet_{L_p\otimes W(\kappa_v)}(\Frob_v\vert D(B)_\bq)^{-1}\\
= & \mydet_{L_p\otimes W(\kappa_v)}(V^{[\kappa_v:\bF_p]}-1\vert D(B)_\bq)\cdot  \mydet_{L_p\otimes W(\kappa_v)}(V^{[\kappa_v:\bF_p]}\vert D(B)_\bq)^{-1}
  \end{align*}
where $V$ denotes the Verschiebung on $D(B)$ and the last identity is \cite{chaioort}[Rem. 10.25]. 
\begin{lemma}There is an exact sequence of $\co_{L_p}\otimes W(\kappa_v)$-modules
$$ 0\to D(B)\xrightarrow{V^{[\kappa_v:\bF_p]}-1}D(B)\to B(\kappa_v)^\wedge\otimes_{\bz_p}W(\kappa_v)\to 0$$
\label{eulerlemma}\end{lemma}
\begin{proof} The kernel of the isogeny $B[p^\infty]\xrightarrow{\Frob_v-1}B[p^\infty]$ of $p$-divisible groups over $\kappa_v$ is the constant finite flat group scheme over $\kappa_v$ associated to the finite abelian $p$-group $B(\kappa_v)[p^\infty]\simeq B(\kappa_v)^\wedge$. The covariant Dieudonn\'e module $D(B(\kappa_v)^\wedge)\simeq B(\kappa_v)^\wedge\otimes_{\bz_p} W(\kappa_v)$ of $B(\kappa_v)^\wedge$ sits in an exact sequence
$$ 0\to B(\kappa_v)^\wedge\otimes_{\bz_p} W(\kappa_v)\to D(B)\otimes\bq_p/\bz_p\xrightarrow{\Frob_v-1}D(B)\otimes\bq_p/\bz_p\to 0$$
by \cite{chaioort}[Prop. 4.53 (ii)]. Multiplication by $p^n$ gives a diagram analogous to (\ref{frob}) and one proceeds as in the case $v\nmid p$. The identity $\Frob_v=V^{[\kappa_v:\bF_p]}$ is again \cite{chaioort}[Rem. 10.25].
\end{proof}
Lemma \ref{eulerlemma} shows that
\begin{align*}& \mydet_{L_p\otimes W(\kappa_v)}(V^{[\kappa_v:\bF_p]}-1\vert D(B)_\bq)\\=&|B(\kappa_v)^\wedge\otimes_{\bz_p}W(\kappa_v)|_{\co_{L_p}\otimes W(\kappa_v)}=|B(\kappa_v)^\wedge|_{\co_{L_p}}.\end{align*}
Similarly, the exact sequence of $\co_{L_p}$-modules \cite{chaioort}[Thm. 4.33 (3)]
$$ 0\to D(B)\xrightarrow{V}D(B)\to \Lie(B) \to 0$$
shows
$$ \mydet_{L_p\otimes W(\kappa_v)}(V^{[\kappa_v:\bF_p]}\vert D(B)_\bq)\sim_{(\co_{L_p}\otimes W(\kappa_v))^\times}\mydet_{L_p}(V\vert D(B)_\bq)=|\Lie(B)|_{\co_{L_p}}$$
proving (\ref{euler3}). 
\end{proof}

 For a perfect complex of $\bz_p$-modules put
$$ M^\dag:=R\Hom_{\bz_p}(M,\bz_p).$$
If $M$ is a finite $\co_{L_p}$-module we have
$$ \mydet_{\co_{L_p}}(M^\dag)=\mydet_{\co_{L_p}}(M^*[-1])=|M^*|_{L_p}\cdot\co_{L_p} \subset \mydet_{{L_p}}(0)=L_p.$$
For $v\mid p$ the isomorphism (\ref{formalgrouplog}) together with the isomorphisms
$$ \hat{\ca}(\fm_v^i)/\hat{\ca}(\fm_v^{i+1})\xrightarrow{\sim}\fm_v^i\mathrm{Lie}(\ca_{\co_{F_v}})/\fm_v^{i+1}\mathrm{Lie}(\ca_{\co_{F_v}})$$
for $i=1,\dots,n-1$ give an equality
\begin{align} \iota_v\left(\mydet_{\co_{L_p}}(\hat{\ca}(\fm_v)) \right) =  & \iota_v\left(\mydet_{\co_{L_p}}(\fm_v\mathrm{Lie}(\ca_{\co_{F_v}}))\right) \notag\\
= & \iota_v\left(\mydet_{\co_{L_p}}(\mathrm{Lie}(\ca_{\co_{F_v}}))\right)\cdot |\mathrm{Lie}(\ca_{\co_{F_v}})\otimes_{\co_{F_v}}\kappa_v|^{-1}_{L_p}\notag\\
= & \iota_v\left(\mydet_{\co_{L_p}}(\mathrm{Lie}(\ca_{\co_{F_v}}))\right)\cdot |\Lie(\ca_{\kappa_v})|^{-1}_{L_p}\notag\\
\stackrel{(\ref{euler2})}{=} & \iota_v\left(\mydet_{\co_{L_p}}(\mathrm{Lie}(\ca_{\co_{F_v}}))\right)\cdot |\ca^0(\kappa_v)^{\wedge}|_{L_p}^{-1}\cdot P_v({^t}\!A/F,Nv^{-1}).\label{euler5}
\end{align}
For $v\nmid p$ we have $\hat{\ca}(\fm_v)^\wedge=0$. Hence
\begin{align*} 
     &\iota_v\left(\mydet_{\co_{L_p}}^{-1} R\Gamma_{/f}(F_v,T)\right)\notag\\
 \stackrel{(\ref{rgmodfderived})}{ = }&\iota_v\left(\mydet_{\co_{L_p}}(A(F_v)^{\wedge,\dag})\right)\\
 \stackrel{(\ref{componentseq})}{ = }& \iota_v\left(\mydet_{\co_{L_p}}(  \ca^0(\co_{F_v})^{\wedge,\dag})\right) \cdot | \Phi_v^{\wedge,*} |_{L_p}\\
 \stackrel{(\ref{formalseq})}{ = }& \iota_v\left(\mydet_{\co_{L_p}}(   \hat{\ca}(\fm_v)^{\wedge,\dag} )\right) \cdot |\ca^0(\kappa_v)^{\wedge,*}|_{L_p} \cdot |\Phi_v^{\wedge,*} |_{L_p}\\
\stackrel{(\ref{euler5})}{  = }& \begin{cases} \iota_v\left(\mydet_{\co_{L_p}}(\mathrm{Lie}(\ca_{\co_{F_v}})^\dag)\right)\cdot P_v({^t}\!A/F,Nv^{-1})\cdot |{^t}\Phi_v|_{L_p} & v\mid p\\
  P_v({^t}\!A/F,Nv^{-1})\cdot |{^t}\Phi_v|_{L_p} & v\nmid p.\end{cases}
\end{align*}
Here we have used the perfect perfect pairing of finite groups \cite{sga7}[IX, 1.3.1]
\begin{equation} {^t}\Phi_v \times \Phi_v\to \bq/\bz.\label{componentpairing}\end{equation} 
The proof of Lemma \ref{localvolume} is now completed by the following isomorphisms
\begin{align*} 
\mathrm{Lie}(\ca_{\co_{F_v}})^\dag \simeq & \Hom_{\bz_p}(\mathrm{Lie}(\ca_{\co_{F_v}})\otimes_{\co_{F_v}}\co_{F_v},\bz_p)\\
\simeq & \Hom_{\co_{F_v}}(\mathrm{Lie}(\ca_{\co_{F_v}}),\Hom_{\bz_p}(\co_{F_v},\bz_p))\\
\simeq & \Hom_{\co_{F_v}}(\mathrm{Lie}(\ca_{\co_{F_v}}),\co_{F_v})\otimes_{\co_{F_v}} \cd^{-1}_{F_v/\bq_p}\\
\simeq & H^0(\, \ca_{\co_{F_v}},\Omega_{\, \ca_{\co_{F_v}}/\co_{F_v}})\otimes_{\co_{F_v}}\cd^{-1}_{F_v/\bq_p}.
\end{align*} 
\end{proof}
We complete the proof of part b) of Prop. \ref{key}. The middle vertical exact triangle in (\ref{wichtig}) and Lemmas \ref{rgf} and \ref{localvolume} give
\begin{align*} &\iota\left(\mydet_{\co_{L_p}}^{-1} R\Gamma(\co_{F,S},T)\right)\\ = &
\iota\left(\mydet_{\co_{L_p}}^{-1} R\Gamma_f(F,T)\otimes \bigotimes\limits_{v\in S}\mydet_{\co_{L_p}}^{-1} R\Gamma_{/f}(F_v,T) \right)\\
= &\frac{|\Sha({^t}\!A/F)|_{L_p}}{|A(F)|_{L_p} |{^t}\! A(F)|_{L_p}}\cdot \prod_v|{^t}\Phi_v|_{L_p}\cdot \prod_{v\in S} P_v({^t}\!A/F,Nv^{-1})\cdot \Upsilon_p
\end{align*} 
using the isomorphism
\begin{align*} H^0(\, \ca,\Omega_{\, \ca/\co_F})\otimes_{\co_F}\cd_{F/\bq}^{-1}\otimes_{\bz}\bz_p\simeq 
&\prod_{v\mid p} H^0(\, \ca,\Omega_{\, \ca/\co_F})\otimes_{\co_F}\cd_{F/\bq}^{-1}\otimes_{\co_F}\co_{F_v}\\
\simeq &\prod_{v\mid p} H^0(\, \ca_{\co_{F_v}},\Omega_{\, \ca_{\co_{F_v}}/\co_{F_v}})\otimes_{\co_{F_v}}\cd^{-1}_{F_v/\bq_p}.
\end{align*}
Since $\Sha({^t}\!A/F)$ and $\Sha(A/F)$, resp. ${^t}\Phi_v$ and $\Phi_v$, are dual finite abelian groups with dual $\co_L$-action we have in fact
$$ |\Sha({^t}\!A/F)|_{L_p}=|\Sha(A/F)|_{L_p};\quad |{^t}\Phi_v|_{L_p}=|\Phi_v|_{L_p}$$
concluding the proof of b).
\end{proof}

For an abelian variety $A/F$ with multiplications by $\co_L\to\End_F(A)$ and each place $v\mid\infty$ of $F$ consider the $\bq$-bilinear $L$-balanced integration pairing
\begin{equation} H_1(A(F_v),\bq) \times H^0(A,\Omega_{A/F})\otimes_FF_v \to F_v\xrightarrow{\mathrm{tr}_{F_v/\br}}\br;\quad  (\gamma,\omega)\mapsto \mathrm{tr}_{F_v/\br}\int_\gamma\omega
\notag\end{equation}
which induces $L_\br$-linear isomorphisms
\begin{align*} \mathrm{per}_v: H^0(A,\Omega_{A/F})\otimes_FF_v& \xrightarrow{\sim} \Hom_\bq(H_1(A(F_v),\bq) ,\br)\xrightarrow{\sim} H^1(A(F_v),\br).\end{align*}
These isomorphisms combine to give a $L_\br$-linear (Deligne) period isomorphism
\begin{equation} \mathrm{per}_{A}: H^0(A,\Omega_{A/F})_\br\simeq \prod_{v\mid\infty}H^1(A(F_v),\br).\label{periodiso}\end{equation}

\begin{definition} For the invertible $\co_L$-module
$$ \Upsilon :=  \mydet_{\co_{L}} \left(H^0(\, \ca,\Omega_{\, \ca/\co_F})\otimes_{\co_F}\cd_{F/\bq}^{-1}\right)$$
introduced in Prop. \ref{key} choose a period  $\Omega\in L_\br^\times$ and a  fractional $\co_L$-ideal $\fa(\Omega)\subset L$ so that
\begin{equation} \mydet_{L_\br}(\mathrm{per}_{A})(\Upsilon)=\Omega \cdot \fa(\Omega)\cdot\mydet_{\co_L}\left(\prod_{v\mid\infty}H^1(A(F_v),\bz)\right)
\label{A-omegadef}\end{equation}
under the determinant of the period isomorphism (\ref{periodiso}).
\label{perioddef}\end{definition} 

Let $A/F$ be a CM abelian variety together with an isomorphism
$$\mu:\co_L\simeq\End_F(A)$$ for a CM field $L$ with $[L:\bq]=2\dim(A)$.
To the CM abelian variety $A/F$ is attached a Serre-Tate character \cite{serretate}[Thm. 10]
$$ \varphi: \ba_F^\times \to L^\times,\quad \varphi\vert_{F^\times}=F^\times\xrightarrow{t}L^{\times}$$
from which a Hecke character  
\begin{equation} \varphi_\tau:\ba_F^\times/F^\times \xrightarrow{\varphi\cdot (t_\br^{-1}\cdot p_\infty)} L_\br^\times\xrightarrow{\tau}\bc^\times\label{hecke}\end{equation}
is deduced for each $\tau\in\Hom(L,\bc)$. Here $p_\infty$ is the projection to $F_\br^\times$ and $t$ is an algebraic homomorphism determined by the CM-type of $A/F$. 
We have the $L_\bc$-valued L-function
$$ L(\varphi,s):=\left(L(\varphi_\tau,s)\right)_\tau\in\prod_\tau\bc\simeq L_\bc$$
which takes values in $L_\br$ for real $s$.  If 
$$\mu':\co_L\simeq \End_F({^t}\!A)$$
 denotes the isomorphism functorially induced by $\mu$ then a polarization $p:A\to{^t}\!A$ induces an isomorphism
$$ (A,\rho\circ\mu)\simeq ({^t}\!A,\mu')$$
of abelian varieties with CM by $L$ (up to isogeny). Here $\rho$ denotes the Rosati involution associated to $p$. Since $\rho$ induces complex conjugation on $L$ the Serre-Tate character of $({^t}\!A,\mu')$ is $\bar{\varphi}$.

\begin{prop} Let $A/F$ be a CM abelian variety so that
$$\co_L\xrightarrow{\sim}\End_F(A)$$ for a CM field $L$ with $[L:\bq]=2\dim(A)$ and assume
$$L(\bar{\varphi},1)\neq 0.$$ 
Let $p$ be any prime number, $T=T_p({^t}\!A)$, $V=T\otimes_{\bz_p}\bq_p$ and $S$ a finite set of places of $F$ containing $\{v\mid p\infty\}$ and all places of bad reduction. Assume that $\Sha(A/F)_{p^\infty}$ and $A(F)$ are finite and let
$$ \iota: \mydet_{L_p}^{-1} R\Gamma(\co_{F,S},V) \simeq  \mydet_{L_p}H^1(\co_{F,S},V) \simeq   \mydet_{L_p} H^0(A,\Omega_{ A/F})\otimes_{\bq}\bq_p$$
be the isomorphism of Prop. \ref{key}. 
Assume there exists $$z\in \mydet _{L_p} H^1(\co_{F,S},V)$$ and a fractional $\co_L$-ideal $\fa(z)$ prime to $p$
with the following properties
\begin{itemize}
\item[a)] $\co_{L_p}\cdot z=  \mydet_{\co_{L_p}}^{-1} R\Gamma(\co_{F,S},T)$
\item[b)] $\iota(z)\in\mydet_L  H^0(A,\Omega_{A/F}) \subset  \mydet_{L_p} \left(H^0(A,\Omega_{ A/F})\otimes_{\bq}\bq_p\right)$
\item[c)] $\co_L\cdot\mydet_{L_\br}(\mathrm{per}_{A})(\iota(z))=L_S(\bar{\varphi},1)\cdot\fa(z)\cdot\mydet_{\co_L}\left(\prod\limits_{v\mid\infty}H^1(A(F_v),\bz)\right)$ 
\end{itemize}
Then $ \frac{L(\bar{\varphi},1)}{\Omega}\in L^\times$ and
$$ \frac{L(\bar{\varphi},1)}{\Omega}=\frac{|\Sha(A/F)|_{L_p}}{|A(F)|_{L_p} |{^t}\! A(F)|_{L_p}}\cdot \prod_v|\Phi_v|_{L_p}\cdot\fa(\Omega)$$
in the group of fractional $\co_L$-ideals supported in $\{\fp\mid p\}$.
\label{keyabelian}\end{prop}

\begin{proof} 

 First note that
\begin{align*} P_v({^t}\!A/F,t)=&\mydet_{L_l}(1-\Frob_v^{-1}\cdot t\vert H^1({^t}\!A_{\bar{F}},\bq_l)^{I_v})\\
=&\mydet_{L_l}(1-\Frob_v\cdot t\vert V_l({^t}\!A)_{I_v})\\
=&1-\overline{\varphi(v)}\cdot t
\end{align*}
and the $L$-equivariant L-function of ${^t}\!A/F$ agrees with $L(\bar{\varphi},s)$.  Also note that $F$ is totally imaginary (since the action of $\co_L$ is defined over $F$) and hence Prop. \ref{key} applies for all primes $p$. By Prop. \ref{key} and a) we have
\begin{align*} \co_{L_p}\cdot \iota(z) =&  \iota\left(\mydet_{\co_{L_p}}^{-1} R\Gamma(\co_{F,S},T)\right)\\
=&\frac{|\Sha(A/F)|_{L_p}}{|A(F)|_{L_p} |{^t}\! A(F)|_{L_p}}\cdot \prod_v|\Phi_v|_{L_p}\cdot \prod_{v\in S} P_v({^t}\!A/F,Nv^{-1})\cdot \Upsilon_p\\
=&\frac{|\Sha(A/F)|_{L_p}}{|A(F)|_{L_p} |{^t}\! A(F)|_{L_p}}\cdot \prod_v|\Phi_v|_{L_p}\cdot \prod_{v\in S} \left(1-\frac{\overline{\varphi(v)}}{Nv}\right)\cdot \Upsilon_p
\end{align*}
and by the definition (\ref{A-omegadef}) of $\Omega$ and $\fa(\Omega)$ we have
\begin{align*}  \co_{L_p}\cdot \mydet_{L_\br}(\mathrm{per_{A}})(\iota(z))
=&\frac{|\Sha(A/F)|_{L_p}}{|A(F)|_{L_p} |{^t}\! A(F)|_{L_p}}\cdot \prod_v|\Phi_v|_{L_p}\cdot \prod_{v\in S} \left(1-\frac{\overline{\varphi(v)}}{Nv}\right)\\
&\cdot\Omega\cdot\fa(\Omega)\cdot\mydet_{\co_L}\left(\prod_{v\mid\infty}H^1(A(F_v),\bz)\right).
 \end{align*}
Comparing this identity with c) we find
$$ L_S(\bar{\varphi},1)= \frac{|\Sha(A/F)|_{L_p}}{|A(F)|_{L_p} |{^t}\! A(F)|_{L_p}}\cdot \prod_v|\Phi_v|_{L_p}\cdot \prod_{v\in S} \left(1-\frac{\overline{\varphi(v)}}{Nv}\right)\cdot\Omega\cdot\fa(\Omega)$$
up to a fractional $\co_L$-ideal $\fa(z)$ prime to $p$. This is the statement of Proposition \ref{keyabelian}.

\end{proof}

Let now $E/F$ be an elliptic curve over a number field $F$ with complex multiplication by $\co_K$ for an imaginary quadratic field $K$. The period $\Omega\in K_\br^\times$ and the fractional ideal $\fa(\Omega)\subseteq K$ defined in the introduction satisfy 
$$\bigotimes\limits_{v\mid\infty} H_1(E(F_v),\bz)=\Omega\cdot\fa(\Omega)\cdot\mydet_{\co_K}\Hom_{\co_F}(H^0(\ce,\Omega_{\ce/\co_F}) ,\co_F)$$
under the determinant over $K_\br$ of the isomorphism
\begin{equation} \prod_{v\mid\infty} H_1(E(F_v),\bq)_\br\cong \Hom_{F}(H^0(E,\Omega_{E/F}),F)_\br\notag\end{equation} 
which is the $\br$-dual of $\mathrm{per}_E$ defined in (\ref{periodiso}). Since 
\begin{align*} \Hom_\bz(H^1(E(F_v),\bz),\bz) \simeq & H_1(E(F_v),\bz)\\
\Hom_\bz(H^0(\ce,\Omega_{\ce/\co_F})\otimes_{\co_F}\cd_{F/\bq}^{-1}
,\bz)\simeq & \Hom_{\co_F}(H^0(\ce,\Omega_{\ce/\co_F}),\co_F)\end{align*}
the quantities $\Omega$ and $\fa(\Omega)$ defined in the introduction coincide with the period and ideal defined in Def. \ref{perioddef}.

\newcommand{\comm}[1]{}

\comm{
The integration pairings (\ref{integration}) also induces $K_\br$-linear isomorphisms
\begin{align*} \mathrm{per}_v: H^0(E,\Omega_{E/F})\otimes_FF_v& \xrightarrow{\sim} \Hom_K(H_1(E(F_v),\bq) ,K_\br)\\
&\xrightarrow{\circ{\mathrm{tr}}_{K_\br/\br}}\Hom_\bq(H_1(E(F_v),\bq) ,\br)\\
&\xrightarrow{\sim} H^1(E(F_v),\br)\end{align*}
for each place $v\mid\infty$ of $F$ which combine to give a $K_\br$-linear isomorphism
\begin{equation} \mathrm{per}: H^0(E,\Omega_{E/F})_\br\simeq \prod_{v\mid\infty}H^1(E(F_v),\br).\label{ellipticperiodiso}\end{equation}
We recall the definition of the period $\Omega$ from Remark \ref{periodremark} which we can normalize so that the corrective fractional ideal $\fa(\Omega)$ is prime to a given prime $p$. Choose elements 
$$\omega\in H^0(\ce,\Omega_{\ce/\co_F});\quad\quad \gamma_v\in H_1(E(F_v),\bz);\quad v\mid\infty$$ 
so that $\omega$ is a $\co_{F_v}$-basis of $H^0(\ce_{\co_{F_v}},\Omega_{\ce_{\co_{F_v}}/\co_{F_v}})$ for $v\mid p$ and $\gamma_v$ is a $\co_{K_p}$-basis of $H_1(E(F_v),\bz_p)$. Choose elements $\beta_1,\dots,\beta_d\in \cd_{F/K}^{-1}$ which form an $\co_{K_v}$-basis of $\cd_{F_v/K_v}^{-1}$ for each place $v\mid p$ of $F$. Define
\begin{equation} \Omega:=\mydet\left(\int_{\gamma_v}\omega\otimes\beta_k\right)_{v,k}. \label{omegadef}\end{equation}

\begin{lemma} Let 
$$ \Upsilon :=  \mydet_{\co_{K}} \left(H^0(\ce,\Omega_{\ce/\co_F})\otimes_{\co_F}\cd_{F/\bq}^{-1}\right)$$
be the invertible $\co_K$-module of Prop.\ref{key} for $A:=E\simeq {^t}\!E$. Then
$$ \mydet_{K_\br}(\mathrm{per})(\Upsilon)=\Omega \cdot \fa(\Omega)\cdot\mydet_{\co_K}\left(\prod_{v\mid\infty}H^1(E(F_v),\bz)\right)$$
where $\fa(\Omega)$ is a fractional $\co_K$-ideal prime to $p$.
\label{periodlemma}\end{lemma}
\begin{proof} We have isomorphisms of invertible $\co_K$-modules
$$ H^1(E(F_v),\bz)\simeq \Hom_\bz(H_1(E(F_v),\bz),\bz)\simeq \Hom_{\co_K}(H_1(E(F_v),\bz)\otimes_{\co_K}\cd_{K/\bq},\co_K).$$
Hence if $\eta_1,\dots,\eta_d\in \cd_{F/\bq}^{-1}$ are elements forming a $\co_{K_v}$-basis of $\cd_{F_v/\bq_p}^{-1}$ for each place $v\mid p$ of $F$ and $\delta\in\cd_{K/\bq}^{-1}$ is a $\co_{K_v}$-basis of $\cd_{K_v/\bq_p}^{-1}$ for each place $v\mid p$ of $K$ then 
$$ \mydet_{K_\br}(\mathrm{per})(\Upsilon)=\Omega' \cdot \fa(\Omega')\cdot\mydet_{\co_K}\left(\prod_{v\mid\infty}H^1(E(F_v),\bz)\right)$$
where $\fa(\Omega')$ is a fractional $\co_K$-ideal prime to $p$ and
$$ \Omega':=\mydet\left(\int_{\gamma_v\otimes\delta^{-1}}\omega\otimes\eta_k\right)_{v,k}.$$
Since by \cite{serre95}
$$ \cd_{F/\bq}^{-1}=\cd_{F/K}^{-1}\cdot \cd_{K/\bq}^{-1}$$
we can in fact take $\eta_k:=\beta_k\cdot\delta$ and by $K$-bilinearity of the integration pairing we find
$$ \Omega'=\mydet\left(\int_{\gamma_v}\omega\otimes\eta_k\cdot\delta^{-1}\right)_{v,k}=\Omega.$$
\end{proof}
}
\begin{prop} Let $E/F$ be an elliptic curve with CM by $\co_K$ and associated Serre-Tate character
$\psi:\ba_F^\times\to K^\times$.
Assume that $$L(\bar{\psi},1)\neq 0.$$ Let $p$ be any prime number, $T=T_p({^t}E)$, $V=T\otimes_{\bz_p}\bq_p$ and $S$ a finite set of places of $F$ containing $\{v\mid p\infty\}$ and all places of bad reduction. Assume that $\Sha(E/F)_{p^\infty}$ and $E(F)$ are finite and let
$$ \iota: \mydet_{K_p}^{-1} R\Gamma(\co_{F,S},V) \simeq  \mydet_{K_p}H^1(\co_{F,S},V) \simeq  \mydet_{K_p} (H^0(E,\Omega_{E/F})\otimes_{\bq}\bq_p)$$
be the isomorphism of Prop. \ref{key}. Assume there exists 
$$z\in \mydet _{K_p} H^1(\co_{F,S},V)$$
and a fractional $\co_K$-ideal $\fa(z)$ prime to $p$ with the following properties
\begin{itemize}
\item[a)] $\co_{K_p}\cdot z=  \mydet_{\co_{K_p}}^{-1} R\Gamma(\co_{F,S},T)$
\item[b)] $\iota(z)\in\mydet_KH^0(E,\Omega_{E/F})\subset \mydet_{K_p}\!\left( H^0(E,\Omega_{E/F})\otimes_\bq\bq_p\right)$
\item[c)] $\co_K\cdot\mydet_{K_\br}(\mathrm{per}_E)(\iota(z))=L_S(\bar{\psi},1)\cdot\fa(z)\cdot\mydet_{\co_K}\left(\prod\limits_{v\mid\infty}H^1(E(F_v),\bz)\right)$ 
\end{itemize}
Then $ \frac{L(\bar{\psi},1)}{\Omega}\in K^\times$ and
$$ \frac{L(\bar{\psi},1)}{\Omega}=\frac{|\Sha(E/F)|_{K_p}}{|E(F)|}\cdot \prod_v|\Phi_v|_{K_p}\cdot\fa(\Omega)$$
in the group of fractional $\co_K$-ideals supported in $\{\fp\mid p\}$.
\label{keyelliptic}\end{prop}

\begin{proof}  This is the special case of Prop. \ref{keyabelian} where $A/F=E/F$ is an elliptic curve and $L=K$, noting that
$$|E(F)|_K\cdot |{^t}E(F)|_K=|E(F)|_K\cdot\overline{|E(F)|}_K=|E(F)|.$$

\comm{
First note that
$$ P_v(E/F,t)=\mydet_{K_l}(1-\Frob_v^{-1}\cdot t\vert H^1(E_{\bar{F}},\bq_l)^{I_v})=1-\bar{\psi}(v)\cdot t$$
and the $K$-equivariant L-function of $E/F$ agrees with $L(\bar{\psi},s)$.  By Prop. \ref{key} and a) we have
\begin{align*} \co_{K_p}\cdot \iota(z) =&  \iota\left(\mydet_{\co_{K_p}}^{-1} R\Gamma(\co_{F,S},T)\right)\\
=& \frac{|\Sha(E/F)|_{K_p}}{|E(F)|_{K_p} |{^t}\! E(F)|_{K_p}}\cdot \prod_v|\Phi_v|_{K_p}\cdot \prod_{v\in S} P_v(E/F,Nv^{-1})\cdot \Upsilon_p\\
=& \frac{|\Sha(E/F)|_{K_p}}{|E(F)|^2_{K_p} }\cdot \prod_v|\Phi_v|_{K_p}\cdot \prod_{v\in S} \left(1-\frac{\bar{\psi}(v)}{Nv}\right) \cdot \Upsilon_p
\end{align*}
and by Lemma \ref{periodlemma}
\begin{align*} & \co_{K_p}\cdot \mydet_{K_\br}(\mathrm{per})(\iota(z))\\
=& \frac{|\Sha(E/F)|_{K_p}}{|E(F)|^2_{K_p} }\cdot \prod_v|\Phi_v|_{K_p}\cdot \prod_{v\in S} \left(1-\frac{\bar{\psi}(v)}{Nv}\right) \cdot 
\Omega\cdot\mydet_{\co_K}\left(\prod\limits_{v\mid\infty}H^1(E(F_v),\bz)\right).
 \end{align*}
Comparing this identity with c) we find
$$ L_S(\bar{\psi},1)= \frac{|\Sha(E/F)|_{K_p}}{|E(F)|^2_{K_p} }\cdot \prod_v|\Phi_v|_{K_p}\cdot \prod_{v\in S} \left(1-\frac{\bar{\psi}(v)}{Nv}\right) \cdot 
\Omega$$
up to a factor in $K^\times$ prime to $p$. This is the statement of Proposition \ref{keyelliptic}.
}

\end{proof}

\section{Kato's reciprocity law}\label{sec:reciprocity}

In this section we recall some definitions and results of \cite{kato00}[\S 15] for which we need to introduce quite a bit of notation. Let $K$ be an imaginary quadratic field and fix an embedding $K\subset \bc$. We identify $\bar{K}=\bar{\bq}$ with the algebraic closure of $K$ in $\bc$. 

\subsection{Iwasawa modules}\label{sec:Iwasawa} For any ideal $\ff$ of $\co_K$ we denote by $$K(\ff)\subseteq\bar{K}$$ the ray class field of conductor $\ff$. For any prime number $p$ and ideal $\ff$ of  $\co_K$ we set
$$ K(p^\infty\ff)=\bigcup_n K(p^n \ff );\quad\quad G_{p^\infty \ff}:=\Gal(K(p^\infty\ff)/K).$$
Then
\begin{equation}G_{p^\infty \ff}\cong \Gamma\times \Delta;\quad \Gamma\simeq\bz_p\times\bz_p \label{deltadef}\end{equation}
where $\Delta:=G_{p^\infty \ff}^{tor}$ is a finite abelian group. Put 
$$\Lambda:=\bz_p[[\Gal(K(p^\infty\ff)/K)]]\simeq\bz_p[\Delta][[T_1,T_2]].$$
Consider the complex of $\Lambda$-modules
$$ R\Gamma_{p^\infty\ff}(\bz_p(1)):=\varprojlim_{K'} R\Gamma(\co_{K'}[\frac{1}{p}], \bz_p(1)) $$
where $K'$ runs through the finite extensions of $K$ contained in $K(p^\infty\ff)$. According to \cite{kato00}[15.6] the cohomology groups of $H^i_{p^\infty\ff}(\bz_p(1))$
are finitely generated $\Lambda$-modules and vanish for $i\neq 1,2$.

\subsection{Elliptic units} \label{sec:elliptic} In \cite{kato00}[15.5] there is defined an elliptic unit
\begin{equation} {_\fa}z_\ff\in\co_{K(\ff)}[1/\ff]^\times\label{elliptic}\end{equation}
for ideals $\fa,\ff$ such that $\co_K^\times\to(\co_K/\ff)^\times$ is injective and $(\fa,6\ff)=1$. If $\ff$ is not a power of a prime ideal then ${_\fa}z_\ff\in\co_{K(\ff)}^\times$. The units ${_\fa}z_\ff$ are norm compatible; in particular for any prime number $p$, any $n\geq 1$ and any nonzero ideal $\ff$ (such that $\co_K^\times\to(\co_K/p^n\ff)^\times$ is injective) one has
$$ N_{K(p^{n+1}\ff)/K(p^n\ff)}({_\fa}z_{p^{n+1}\ff})={_\fa}z_{p^n\ff}.$$
Denoting by $(\fa,F/K)\in\Gal(F/K)$ the Artin symbol, we have in particular an element 
$$\sigma_\fa:=(\fa,K(p^\infty\ff)/K)\in G_{p^\infty\ff}\subset\Lambda^\times.$$ 
Define
\begin{align*}z_{p^\infty\ff} :=& \left(N\fa-\sigma_\fa\right)^{-1}({_\fa}z_{p^n\ff})_{n\geq 1}\\
 \in &\left(\varprojlim_{K'} \co_{K'}[\frac{1}{p}]^\times\otimes_\bz\bz_p\right)\otimes_\Lambda Q(\Lambda)\simeq 
 H^1_{p^\infty\ff}(\bz_p(1))\otimes_\Lambda Q(\Lambda)\end{align*}
which is  independent of $\fa$.

\subsection{Hecke characters}\label{sec:hecke} Let $$\varphi:\ba_K^\times\to L^\times$$ be an algebraic Hecke character of $K$ with values in the number field $L$ and of infinity type $(-1,0)$. 
Following \cite{kato00}[15.8] we recall the definition of the motivic structure associated to $\varphi$. This consists of rank one $L$-vector spaces $V_L(\varphi)$ and $S(\varphi)$, a continuous $L_\csp$-linear $\Gal(\bar{\bq}/K)$-representation $V_{L_\csp}(\varphi)$ for each place $\csp\mid p$ of $L$ together with a (Deligne) period isomorphism
 $$ \mathrm{per}_\tau:S(\varphi)\otimes_{L,\tau}\bc\xrightarrow{\simeq} V_L(\varphi)\otimes_{L,\tau}\bc$$
for each embedding $\tau:L\to\bc$ and comparison isomorphisms
\begin{equation} V_{L_\csp}(\varphi)\simeq V_L(\varphi)\otimes_LL_\csp\label{artin}\end{equation}
as well as $p$-adic (Deligne) period isomorphisms
\begin{equation}D^1_{dR}(K_p,V_{L_\csp}(\varphi))\simeq S(\varphi)\otimes_LL_{\csp}\label{padicperiod}\end{equation}
for each $\csp\mid p$.

Let $\ff$ be a multiple of the conductor of $\varphi$ such that $\co_K^\times\to (\co_K/\ff)^\times$ is injective, and let $E=(E,\alpha)$ be the canonical CM-pair over $K(\ff)$  in the sense of \cite{kato00}[(15.3.1)], i.e. $E/K(\ff)$ is an elliptic curve with CM by $\co_K$ and $\alpha\in E(K(\ff))$ is a torsion point with annihilator $\ff$. As explained in \cite{kato00}[(15.3.3)]
if $\fa$ is an ideal prime to $\ff$ with Artin symbol $\sigma=(\fa,K(\ff)/K)\in\Gal(K(\ff)/K)$ there is a canonical isomorphism 
$$ \eta_\fa: (E/E[\fa], \alpha\  \mathrm{mod}\  E[\fa])\simeq (E^{(\sigma)},\sigma(\alpha)).$$ 
We denote by $\eta_\fa^*$ the map induced on cohomology by the composite isogeny $E\to E/E[\fa]\xrightarrow{\eta_\fa}E^{(\sigma)}$.
We then define
\begin{align*} V_L(\varphi):= & H^1(E(\bc),\bq)\otimes_KL\\
S(\varphi):= &(H^0(E,\Omega_{E/K(\ff)})\otimes_KL)^{\Gal(K(\ff)/K)}
\end{align*}
where $\sigma\in \Gal(K(\ff)/K)$ acts as the composite
$$ H^0(E,\Omega_{E/K(\ff)})\underset{K}{\otimes}L\xrightarrow{\sigma\otimes 1} H^0(E^{(\sigma)},\Omega_{E^{(\sigma)}/K(\ff)})\underset{K}{\otimes}L\xrightarrow[\simeq]{\varphi(\fa)^{-1}\eta_\fa^*}H^0(E,\Omega_{E/K(\ff)})\underset{K}{\otimes}L.$$
For each place $\csp\mid p$ of $L$ we define a $\Gal(\bar{\bq}/K)$-representation
$$V_{L_\csp}(\varphi):=  H^1_{et}(E\otimes_{K(\ff)}\bar{\bq},\bq_p)\otimes_{K_p}L_{\csp}$$
where $\sigma\in \Gal(\bar{\bq}/K)$ acts via 
$$V_{L_\csp}(\varphi)=H^1_{et}(E\otimes_{K(\ff)}\bar{\bq},\bq_p)\underset{K_p}{\otimes}L_{\csp}\xrightarrow{\sigma} H^1_{et}(E^{(\sigma)}\otimes_{K(\ff)}\bar{\bq},\bq_p)\underset{K_p}{\otimes}L_{\csp}\xrightarrow{\varphi(\fa)^{-1}\eta_\fa^*} V_{L_\csp}(\varphi).$$
Here $\fa$ an ideal such that $\sigma\vert_{K(\ff)}=(\fa,K(\ff)/K)$. The isomorphism (\ref{padicperiod}) is induced by the $p$-adic period isomorphism for $E/K(\ff)$ \cite{kato00}[(15.8.1)] and the isomorphism $\mathrm{per}_\tau$ is induced by the period isomorphism (\ref{periodiso}) for $E/K(\ff)$.
\begin{remark}
In the construction of the motivic structure the role of a torsion point $\alpha\in E(K(\ff))$ is to fix the isomorphism $\eta_{\fa}: E/E[\fa]\simeq E^{(\sigma_{\fa})}$.  
It induces the isogeny 
$E\to E/E[\fa]\xrightarrow{\eta_\fa}E^{(\sigma_\fa)}$ 
which is the only way $\eta_\fa$ enters into the construction. Note that the isogeny $E\rightarrow E^{(\sigma_{\fa})}$ is uniquely determined. 
\label{ind}\end{remark}

\subsection{The reciprocity law}\label{sec:erl} We state Kato's reciprocity law and then deduce its consequences for an elliptic curve $E/F$ as in Thm. \ref{main}.

\begin{prop} Let $\varphi$ be an algebraic Hecke character of $K$ with values in the number field $L$ and of infinity type $(-1,0)$. 
For an embedding $\tau:L\to\bc$ let  
$$\varphi_\tau:\ba_K^\times/K^\times\to\bc^\times$$
be the Hecke character deduced from $\varphi$ as in  (\ref{hecke}).
Let $\csp\mid p$ be any prime ideal of $\co_L$, $\ff$ a multiple of the conductor of $\varphi$ and $\gamma\in V_L(\varphi)$. Then the image $z_{p^\infty\ff}(\gamma)'$ of $z_{p^\infty\ff}$ under
\begin{align*} &H^1_{p^\infty\ff}(\bz_p(1))\xrightarrow{\gamma}H^1_{p^\infty\ff}(\bz_p(1))\otimes V_{L_\csp}(\varphi)\simeq H^1_{p^\infty\ff}(V_{L_\csp}(\varphi)(1))\\
\to & H^1(\co_K[\frac{1}{p}],V_{L_\csp}(\varphi)(1)) \xrightarrow{\exp^*}D^1_{dR}(K_p,V_{L_\csp}(\varphi))\simeq S(\varphi)\otimes_LL_{\csp}
 \end{align*}
 is an element of $S(\varphi)$. Moreover
 $$\mathrm{per}_\tau(z_{p^\infty\ff}(\gamma)')= L_{p\ff}(\bar{\varphi}_\tau,1)\cdot\gamma.$$
\label{reciprocity}\end{prop}

\begin{proof} This is the special case of \cite{kato00}[Prop. 15.9] where $\varphi$ has infinity type $(-1,0)$ and where $K'=K$.
\end{proof}

\begin{remark} Prop. \ref{reciprocity} includes Deligne's period conjecture \cite{deligne79} for the algebraic Hecke character $\bar{\varphi}$. If $b$ is an $L$-basis of $S(\varphi)$ and
$$\Omega=\Omega(b,\gamma)=(\Omega_\tau)\in L_\br^\times$$ is such
$\mathrm{per}_\tau(b)=\Omega_\tau\cdot\gamma$ for all $\tau$ then
$$\frac{L_{p\ff}(\bar{\varphi},1)}{\Omega}\in L\subseteq  L_\br.$$
In particular, if $L_{p\ff}(\bar{\varphi}_{\tau_0},1)\neq 0$ for one $\tau_0$ then
$$L_{p\ff}(\bar{\varphi}_\tau,1)\neq 0$$ for all $\tau\in\Hom(L,\bc)$.
Deligne's period conjecture in the situation of Prop. \ref{reciprocity} was proven in \cite{gos} and is known for all algebraic Hecke characters of all number fields (see \cite{schappacher}[Ch. II, Thm. 2.1] and references therein. The proof for non-CM base fields was recently completed in \cite{kufner}).
\label{rem:algebraic}\end{remark}

Recall the following proposition from \cite{gos}[Thm. 4.1]

\begin{prop} Let $E/F$ be an elliptic curve over a number field $F$ with complex multiplication by the ring of integers in an imaginary quadratic field $K$. Then $K\subseteq F$ and the following are equivalent
\begin{itemize}
\item[a)] $F(E_{tors})/K$ is an abelian extension of $K$.
\item[b)] The abelian variety $B:=\Res_{F/K} E$ has complex multiplication over $K$ in the sense that
$$ L:=\End_K(B)\otimes\bq\simeq  L_1\times\cdots\times L_r$$
where $L_1,\dots,L_r$ are CM fields containing $K$ such that
$$ [L:K]=\sum_{i=1}^r[L_i:K]=[F:K] (=\dim B).$$
\item[c)] The extension $F/K$ is abelian and there exists an algebraic Hecke character $\eta$ of $K$ so that
$$ \psi=\eta\circ N_{F/K} $$
where $\psi$ is the algebraic Hecke character of $F$ associated to $E/F$.
\end{itemize}
\label{CM}\end{prop}

To the CM abelian variety $B/K$ is attached a $L$-valued Serre-Tate character \cite{serretate}[Thm. 10]
$$ \varphi=(\varphi_1,\dots,\varphi_r): \ba_K^\times \to L^\times,\quad \varphi\vert_{K^\times}=K^\times\xrightarrow{i}L^{\times}$$
where $\varphi_j$ is the Serre-Tate character of the simple isogeny factor $B_j$ of $B$ with endomorphism algebra $L_j$. Here $i$ is the inclusion.

\begin{prop} In the situation of Prop. \ref{CM} there are isomorphisms of free rank one $L$-modules
\begin{align*} V_{L}(\varphi):=V_{L_1}(\varphi_1)\times\cdots\times V_{L_r}(\varphi_r)\simeq \,& H^1(B(\bc),\bq)\\
S(\varphi):=S(\varphi_1)\times\cdots\times S(\varphi_r)\simeq\,  &H^0(B,\Omega_{B/K})
\end{align*}
so that the diagram of free rank one $L_\br$-modules
$$\begin{CD} 
\prod\limits_{\tau\in\Hom_K(L,\bc)} S(\varphi)\otimes_{L,\tau}\bc @>\simeq>> S(\varphi)_\br @>\simeq >>H^0(B,\Omega_{B/K})_\br\\
@V\prod_\tau \mathrm{per}_{\tau}VV @. @VV\mathrm{per}_B V\\
\prod\limits_{\tau\in\Hom_K(L,\bc)} V_{L}(\varphi)\otimes_{L,\tau}\bc @>\simeq>> V_{L}(\varphi)_\br @>\simeq>> H^1(B(\bc),\bq)_\br 
\end{CD}$$
commutes where $\mathrm{per}_B$ was defined in (\ref{periodiso}). Moreover, for each prime number $p$ there is a $\Gal(\bar{\bq}/K)$-equivariant isomorphism of free rank one $L_p$-modules
$$ V_p(\varphi):=\prod_{\csp\mid p}V_{L_{1,\csp}}(\varphi_1)\times\cdots\times \prod_{\csp\mid p}V_{L_{r,\csp}}(\varphi_r)\simeq H^1_{et}(B\otimes_{K}\bar{\bq},\bq_p)$$
compatible with (\ref{artin}) and the Artin comparison isomorphism for $B$. Finally, the $p$-adic (Deligne) period isomorphism for $B$ is compatible with (\ref{padicperiod}).
\label{bcomp}\end{prop}

\begin{proof}
The following is based on the construction of $B/K$ via Galois descent. 
Put $G=\Gal(F/K)$ and define an abelian variety 
$$
\tilde{B}=\prod_{\sigma' \in G} E^{(\sigma')}
$$
for $E^{(\sigma')}$ the Galois conjugate. An element $\sigma \in G$ induces an isomorphism $E^{(\sigma')}\simeq E^{(\sigma\circ \sigma')}$ which leads to $$\phi_{\sigma}: \tilde{B} \simeq \tilde{B}^{(\sigma)}.$$
Note that $(\tilde{B},(\phi_{\sigma})_{\sigma \in G})$ is an effective descent datum, $B/K$ being the descent. 
One has
$$
\End_{F}(\tilde{B})^{G}=\prod_{\sigma'\in G}\Hom_{F}(E,E^{(\sigma')})
$$
and accordingly Prop. \ref{CM} b) gives a partition of 
$G$ by the indices $\{1,...,r\}$. 
Let $G_{i}$ denote the subset of elements associated to an index $i$. 
For each $i$ define an abelian variety  
\begin{equation}\label{partial}
\tilde{B}_i=\prod_{\tau\in G_{i}} E^{(\tau)}.
\end{equation}
The descent datum on $\tilde{B}$ induces a datum on $\tilde{B}_i$, let $B_i/K$ denote the descent. 
Note that there is an isogeny 
\begin{equation}\label{iso}
B \rightarrow \prod_{i=1}^{r} B_i
\end{equation}
over $K$. 
The main theorem of complex multiplication leads to the following description of the descent datum on $\tilde{B}_i$. 
Let $\co_i\subset L_i$ denote the endomorphism ring of $B_i$. 
For $\sigma \in G$ 
pick $s_{\sigma} \in \ba_{K,f}^\times$ with ${\rm rec}_{K}(s_{\sigma})=\sigma$
where 
$${\rm rec}_{K}: \ba_{K}^{\times}/K^{\times}\simeq \Gal(K^{ab}/K)\twoheadrightarrow G$$
 is the Artin map normalised so that  uniformisers map to lifts of the the arithmetic Frobenius. 
 By the main theorem of complex multiplication 
 \cite{cco}[Thm.~A.2.7]
 there is a unique $L_i$-linear isomorphism of abelian varieties 
$$
\theta_{\sigma,s_{\sigma}}: \tilde{B}_i\otimes_{\co_i} I_{s_{\sigma}}
\simeq \tilde{B}_i^{(\sigma)} 
$$
for $I_{s_{\sigma}}$ the principal fractional $\co_i$-ideal 
generated by $\varphi_i(s_{\sigma})^{-1}\in L_i^\times$. 
The composite 
\begin{equation}\label{dsc}
c(\sigma): \tilde{B}_i \stackrel{\varphi_i(s_{\sigma})^{-1}}{\simeq} \tilde{B}_i\otimes_{\co_i} I_{s_{\sigma}} \stackrel{\theta_{\sigma,s_{\sigma}}}{\simeq} \tilde{B}_i^{(\sigma)}
\end{equation}
is an $L_i$-linear $F$-isomorphism. 
For varying $\sigma$ the isomorphisms $c(\sigma): \tilde{B}_i\simeq \tilde{B}_i^{(\sigma)}$ induce an 
$\co_i$-linear descent datum on $\tilde{B}_i$ with respect to $G$, which is compatible with the preceding datum \cite{cco}[A.3.4] (see also \cite{shimura71}[p. 513]). 

Note that the motivic structure associated to a Hecke character $\varphi$ as in section \ref{sec:hecke} may be defined via a CM pair $(E',\alpha')$ where $E'/F'$ is an elliptic curve as in Prop. \ref{CM} and $\alpha'\in E'(\tilde{F}')$ is a torsion point with annihilator $\ff$ a multiple of the conductor of $\varphi$ and 
$\tilde{F}'/K$ an abelian extension containing $F'$. 
The resulting motivic structure is independent of the choice \cite{kato00}[p. 257].
In light of Remark \ref{ind} the elliptic curve $E'/F'$ along with the isogeny $E'\rightarrow E^{'\sigma_{\fa}}$ for 
$\sigma_{\fa}\in\Gal(F'/K)$ give rise to the motivic structure. 
In the following we may thus consider an elliptic curve $E^{(\tau)}/F$ as above for $\tau\in G_i$. 
By definition 
\begin{equation}\label{cp-1}
H^{1}(B_i(\bc),\bq)=H^{1}(\tilde{B}_i(\bc),\bq)\simeq H^{1}(E^{(\tau)}(\bc),\bq)\otimes_{K} L_i=V_{L_i}(\varphi_i).
\end{equation}
As for the de Rham realisation $S(\varphi_i)$ first note 
$$H^{0}(\tilde{B}_i,\Omega_{\tilde{B}_i/F})\simeq H^0(E^{(\tau)},\Omega_{E^{(\tau)}/F})\otimes_{K} L_i$$
since the endomorphism ring of $\tilde{B}_i$ is an order in $L_i$ and \eqref{partial}.
In light of the construction of $B_i/K$ observe $H^{0}(B_i,\Omega_{B_i/K})$ is the fixed part of the $\Gal(F/K)$-action on 
$H^{0}(\tilde{B}_i,\Omega_{\tilde{B}_i/F})$ arising from the descent datum \eqref{dsc}. From the above description the action coincides with the $\Gal(F/K)$-action on $H^0(E^{(\tau)},\Omega_{E^{(\tau)}/F})\otimes_{K} L_i$ as in section \ref{sec:hecke}. Hence one has 
\begin{equation}\label{cp-2}
H^{0}(B_i,\Omega_{B_i/K})\simeq S(\varphi_i).
\end{equation}
In the same vein the construction induces an isomorphism of $L_{i,\csp}[G_{K}]$-modules 
\begin{equation}\label{cp-3}
H^{1}_{et}(B_i\otimes_{K}\bar{K}_{w},\bq_{p})\otimes_{L_i\otimes\bq_{p}}L_{i,\csp} \simeq V_{L_{i,\csp}}(\varphi_i). 
\end{equation}
Under the isomorphisms \eqref{cp-1}, \eqref{cp-2} and \eqref{cp-3} note that 
 the period maps ${\rm{per}}_\tau$ and \eqref{padicperiod} as in section \ref{sec:hecke} correspond to the period maps
$$
{\rm{per}_{B_i}}: H^{0}(B_i,\Omega_{B_i/K}) \rightarrow H^{1}(B_i(\bc),\bc)
$$
and 
$$
D_{dR}(L_i\otimes \bq_{p},H^{1}_{et}(B_i\otimes_{K}\bar{K}_{v},\bq_{p})\otimes_{L_{i}\otimes\bq_{p}}L_{i,\csp})\simeq H^{1}_{dR}(B_{\tau}/K_{v}) \otimes_{L_{i}\otimes\bq_{p}} L_{i, \csp} 
$$
respectively. In view of the isogeny \eqref{iso} the proof concludes. 
\end{proof}

\comm{
\begin{proof}
The following is based on the construction of $B/K$ via Galois descent. 
Put $G=\Gal(F/K)$ and define an abelian variety 
$$
\tilde{B}=\prod_{\sigma' \in G} E^{(\sigma')}
$$
for $E^{(\sigma')}$ the Galois conjugate. An element $\sigma \in G$ induces an isomorphism $E^{(\sigma')}\simeq E^{(\sigma\circ \sigma')}$ which leads to $$\phi_{\sigma}: \tilde{B} \simeq \tilde{B}^{(\sigma)}.$$
Note that $(\tilde{B},(\phi_{\sigma})_{\sigma \in G})$ is an effective descent datum, $B/K$ being the descent. 
One has
$$
\End_{F}(\tilde{B})^{G}=\prod_{\sigma'\in G}\Hom_{F}(E,E^{(\sigma')})
$$
and accordingly Prop. \ref{CM} c) gives a partition of 
$G$ by the indices $\{1,...,r\}$. 
Let $G_{i}$ denote the subset of elements associated to an index $i$. 
For $\tau\in\Hom_{K}(L,\bc)$ define an abelian variety  
$$
\tilde{B}_{\tau}=\prod_{\tau'\in G_{i(\tau)}} E^{(\tau')}.
$$
The descent datum on $\tilde{B}$ induces a datum on $\tilde{B}_\tau$, let $B_{\tau}/K$ denote the descent. 
Note that there is an isogeny 
\begin{equation}\label{iso}
B \rightarrow \prod_{i=1}^{r} B_{\tau_{i}}
\end{equation}
over $K$. 
The main theorem of complex multiplication leads to the following description of the descent datum on $\tilde{B}_\tau$. 
Let $\co_{i(\tau)}\subset L_{i(\tau)}$ denote the endomorphism ring of $B_{\tau}$. 
For $\sigma \in G$ 
pick $s_{\sigma} \in \ba_{K,f}^\times$ with $rec_{K}(s_{\sigma})=\sigma$
where $rec_{K}: \ba_{K}^{\times}/K^{\times}\simeq \Gal(K^{ab}/K)\twoheadrightarrow G$ is the Artin map normalised so that  uniformisers map to lifts of the the arithmetic Frobenius. 
 By the main theorem of complex multiplication 
 \cite{cco}[Thm.~A.2.7]
 there is a unique $L_{i(\tau)}$-linear isomorphism of abelian varieties 
$$
\theta_{\sigma,s_{\sigma}}: \tilde{B}_{\tau}\otimes_{\co_{i(\tau)}} I_{s_{\sigma}}
\simeq \tilde{B}_{\tau}^{(\sigma)} 
$$
for $I_{s_{\sigma}}$ the principal fractional ideal 
of $\co_{i(\tau)}$ generated by $\varphi_{\tau}(s_{\sigma})^{-1}\in L_{i(\tau)}^\times$. 
The composite 
\begin{equation}\label{dsc}
c(\sigma): \tilde{B}_{\tau} \stackrel{\varphi_{\tau}(s_{\sigma})^{-1}}{\simeq} \tilde{B}_{\tau}\otimes_{\co_{i(\tau)}} I_{s_{\sigma}} \stackrel{\theta_{\sigma,s_{\sigma}}}{\simeq} \tilde{B}_{\tau}^{(\sigma)}
\end{equation}
is an $L_{i(\tau)}$-linear $F$-isomorphism. 
For varying $\sigma$ the isomorphisms $c(\sigma): \tilde{B}_{\tau}\simeq \tilde{B}_{\tau}^{(\sigma)}$ induce an 
$\co_{i(\tau)}$-linear descent datum on $\tilde{B}_{\tau}$ with respect to $G$, which is compatible with the preceding datum \cite{cco}[A.3.4] (see also \cite{shimura71}[p. 513]). 

Note that the motivic structure associated to a Hecke character $\varphi$ as in section \ref{sec:hecke} maybe defined via a CM pair $(E',\alpha')$ where $E'/F'$ is an elliptic curve as in Prop. \ref{CM} and $\alpha'\in E'(\tilde{F}')$ is a torsion point with annihilator $\ff$ a multiple of the conductor of $\varphi$ and 
$\tilde{F}'/K$ an abelian extension containing $F'$. 
The resulting motivic structure is independent of the choice \cite{kato00}[p. 257].
In light of Remark \ref{ind} the elliptic curve $E'/F'$ along with the isogeny $E'\rightarrow E^{'\sigma_{\fa}}$ for 
$\sigma_{\fa}\in\Gal(F'/K)$ give rise to the motivic structure. 
In the following we may thus consider an elliptic curve $E^{(\tau')}/F$ as above for $\tau'\in G_{i(\tau)}$. 
By definition 
\begin{equation}\label{cp-1}
H^{1}(B_{\tau}(\bc),\bq)=H^{1}(\tilde{B}_{\tau}(\bc),\bq)\simeq H^{1}(E^{(\tau')}(\bc),\bq)\otimes_{K} L_{i(\tau)}=V_{L_{i(\tau)}}(\varphi_{\tau}).
\end{equation}
As for the de Rham realisation $S(\varphi_{\tau})$ first note 
$$H^{0}(\tilde{B}_{\tau},\Omega_{\tilde{B}_{\tau}/F})\simeq H^0(E^{(\tau')},\Omega_{E^{(\tau')}/F})\otimes_{K} L_{i(\tau)}.$$
In light of the construction of $B_{\tau}/K$ observe $H^{0}(B_{\tau},\Omega_{B_{\tau}/K})$ is the fixed part of the $\Gal(F/K)$-action on 
$H^{0}(\tilde{B}_{\tau},\Omega_{\tilde{B}_{\tau}/F})$ arising from the descent datum \eqref{dsc}. From the above description the action coincides with the $\Gal(F/K)$-action on $H^0(E^{(\tau')},\Omega_{E^{(\tau')}/F})\otimes_{K} L_{i(\tau)}$ as in section \ref{sec:hecke}. Hence one has 
\begin{equation}\label{cp-2}
H^{0}(B_{\tau},\Omega_{B_{\tau}/K})\simeq S(\varphi).
\end{equation}
In the same vein the construction induces an isomorphism of $L_{i(\tau),\csp}[G_{K}]$-modules 
\begin{equation}\label{cp-3}
H^{1}_{et}(B_{\tau}\otimes_{K}\bar{K}_{w},\bq_{p})\otimes_{L_{i(\tau)}\otimes\bq_{p}}L_{i(\tau),\csp} \simeq V_{L_{i(\tau),\csp}}(\varphi_{\tau}). 
\end{equation}
Under the isomorphisms \eqref{cp-1}, \eqref{cp-2} and \eqref{cp-3} note that 
 the period maps ${\rm{per}}_{\varphi}$ and \eqref{padicperiod} as in section \ref{sec:hecke} correspond to the period maps
$$
{\rm{per}}: H^{0}(B_{\tau},\Omega_{B_{\tau}/K}) \rightarrow H^{1}(B_{\tau}(\bc),\bc)
$$
and 
$$
D_{dR}(L_{i(\tau)}\otimes \bq_{p},H^{1}_{et}(B_{\tau}\otimes_{K}\bar{K}_{v},\bq_{p})\otimes_{L_{i(\tau)}\otimes\bq_{p}}L_{i(\tau),\csp})\simeq H^{1}_{dR}(B_{\tau}/K_{v}) \otimes_{L_{i(\tau)}\otimes\bq_{p}} L_{i(\tau), \csp} 
$$
respectively. In view of the isogeny \eqref{iso} the proof concludes. 
\end{proof}
}
\begin{corollary} In the situation of Prop. \ref{CM} let $p$ be any prime number, $\ff$ a multiple of the conductor of $B$ and $\gamma\in H^1(B(\bc),\bq)$. Then the image $z_{p^\infty\ff}(\gamma)'$ of $z_{p^\infty\ff}$ under
\begin{align*} &H^1_{p^\infty\ff}(\bz_p(1))\xrightarrow{\gamma}H^1_{p^\infty\ff}(\bz_p(1))\otimes  H^1_{et}(B\otimes_{K}\bar{\bq},\bq_p) \simeq H^1_{p^\infty\ff}(V_p({^t}B))\\
\to & H^1(\co_K[\frac{1}{p}],V_p({^t}B)) \xrightarrow{\exp^*}H^0(B_{K_p},\Omega_{B_{K_p}/K_p})
 \end{align*}
is an element $H^0(B,\Omega_{B/K})$. Moreover, if $\gamma_1,\dots,\gamma_d$ is a $K$-basis of $H^1(B(\bc),\bq)$ then 
\begin{equation} \mydet_{K_\br}(\mathrm{per}_B)\left(z_{p^\infty\ff}(\gamma_1)'\wedge\cdots\wedge z_{p^\infty\ff}(\gamma_d)'\right)=L_{p\ff}(\bar{\psi},1)\cdot(\gamma_1\wedge\cdots\wedge\gamma_d).\label{periodlvalue}\end{equation}
\label{breciprocity}\end{corollary}

\begin{proof} The first statement is clear from Prop. \ref{reciprocity} for $\varphi_1,\dots,\varphi_r$. Since $\gamma\mapsto z_{p^\infty\ff}(\gamma)$ is $K$-linear, and its scalar extension $K_\br$-linear, it suffices to show (\ref{periodlvalue}) for a particular $K_\br$-basis $\{\gamma_i\}$ of $H^1(B(\bc),\bq)_\br$ in order to deduce it for all. Taking $\{\gamma_i\}=\{\gamma_\tau\}$ where $\gamma_\tau$ is a $K_\br=\bc$-basis of $V_{L}(\varphi)\otimes_{L,\tau}\bc$ Prop. \ref{reciprocity} gives the equality
$$ \mydet_{K_\br}(\mathrm{per}_B)\left(z_{p^\infty\ff}(\gamma_1)\wedge\cdots\wedge z_{p^\infty\ff}(\gamma_d)\right)=\prod\limits_{\tau\in\Hom_K(L,\bc)}L_{p\ff}(\bar{\varphi}_{\tau},1)\cdot(\gamma_1\wedge\cdots\wedge\gamma_d).$$
It remains to recall the identity of L-functions \cite{gos}[Eq. (5.0), Lemma (4.8)(iii)]
\begin{equation} L_{p\ff}(\bar{\psi}_\iota,s)=\prod_{\tau\vert_K=\iota} L_{p\ff}(\bar{\varphi}_\tau,s)\label{norm}\end{equation}
where $\iota:K\to \bc$ is the embedding fixed above.
One can view the left hand side as the $K$-equivariant L-function of ${^t}E/F$ or, since 
$$ V_p({^t}B)=\Ind_{G_F}^{G_K} V_p({^t}E),$$
as the $K$-equivariant L-function of ${^t}B=\Res_{F/K}{^t}E$ over $K$. On the other hand, the tuple 
$$(L_{p\ff}(\bar{\varphi}_\tau,s))_\tau\in \prod_{\tau\vert_K=\iota} \bc\simeq L_\br$$ 
can be viewed as the  $L$-equivariant L-function of ${^t}B/K$. The identity (\ref{norm}) then amounts to the fact that the norm from $L_\br$ to $K_\br$ of the $L$-equivariant L-function is the $K$-equivariant L-function.
\end{proof}

\section{The Iwasawa main conjecture}\label{iwasawa}

In section \ref{twistedelliptic} we recall the exact notation for the Euler system of elliptic units used in \cite{jlk} and match it with the notation already introduced in section \ref{sec:elliptic} (which is identical to Kato's notation in \cite{kato00}). In section \ref{sec:mc} we recall the "$\Lambda$-main conjecture" of \cite{jlk} associated to an arbitrary prime number $p$ and finite order character $\chi$ of $G_K$. In section \ref{sec:descent} we compute the image of the basis given by the main conjecture in the determinant of Galois cohomology of the Galois representation associated to a Hecke character. This will allow us to complete the proof of Thm. \ref{main2}, resp. Thm. \ref{main}, in section \ref{sec:ab}, resp. \ref{sec:ell}.
 
\subsection{Twisted Elliptic Units}\label{twistedelliptic} We use the notation of sections \ref{sec:Iwasawa} and \ref{sec:elliptic}. Let $\co$ be the ring of integers in a finite extension of $\bq_p$ and
$$ G_K\to G_{p^\infty\ff}\xrightarrow{\chi} \co^\times$$
a finite order character of conductor $\ff_\chi\mid\ff$. 
Following \cite{jlk}[Def. 1.1] we denote by $\co(\chi)$ the free rank one $\co$-module on which $G_{p^\infty\ff}$ acts via $\chi^{-1}$ and following \cite{jlk}[Def. 4.2] we define
\begin{equation} \Lambda(\chi):=\co(\chi)\otimes_{\bz_p}\bz_p[[\Gamma]].\label{lchidef}\end{equation}
Then $\Lambda(\chi)$ is a free, rank one module over
$$\Lambda_\co:=\co[[\Gamma]] \simeq\co[[T_1,T_2]]$$
with a continuous $\Lambda_\co$-linear $G_{p^\infty\ff}$-action.

For nonzero ideals $\fa,\fm$ of $\co_K$, prime number $p$ and $n\geq 1$ so that $(\fa,6p\fm)=1$ and $\co_K^\times\to(\co_K/p^n)^\times$ is injective, define \cite{jlk}[Def. 3.2]
$$ {_\fa}\zeta_\fm:=N_{K(p^n\fm)/K(\fm)} ({_\fa}z_{p^n\fm}).$$
Note that ${_\fa}\zeta_\fm$ depends on $p$ (in addition to $\fa$ and $\fm$) but not on $n$. For an $\co$-basis $t(\chi)$ of $\co(\chi)$ define \cite{jlk}[Def. 3.5]
$$ {_\fa}\zeta_\fm(\chi):=\trace_{K(\ff_\chi\fm)/K(\fm)}({_\fa}\zeta_{\ff_\chi\fm}\otimes t(\chi))\in H^1(\co_{K(\fm)}[\frac{1}{p}], \co(\chi)(1)).$$ 
Here $\co(\chi)$ denotes the $p$-adic \'etale sheaf $j_*\co(\chi)$ where 
$$j:\Spec\co_{K(\fm)}[\frac{1}{p\ff}]\to\Spec\co_{K(\fm)}[\frac{1}{p}]$$ is an open embedding and the Galois module $\co(\chi)$ is viewed as a local system on $\Spec\co_{K(\fm)}[\frac{1}{p\ff}]$.
For any field $K\subseteq F\subseteq K(\fm)$ define
$$ {_\fa}\zeta_F(\chi):=\trace_{K(\fm)/F}({_\fa}\zeta_\fm(\chi)).$$
Denote by $K_n/K$ the fixed field of the kernel of $G_{p^\infty \ff}\to\Gamma\to \Gamma/\Gamma^{p^n}$ and define \cite{jlk}[5.2]
$${_\fa}\zeta(\chi):=\varprojlim_n {_\fa}\zeta_{K_n}(\chi)\in H^1(\co_{K}[\frac{1}{p}], \Lambda(\chi)(1))$$
and
$$\zeta(\chi):=(N\fa-\sigma_\fa)^{-1}{_\fa}\zeta(\chi)\in H^1(\co_{K}[\frac{1}{p}], \Lambda(\chi)(1))\otimes_{\Lambda_\co}Q(\Lambda_\co).$$
From section \ref{sec:elliptic} recall the norm compatible system
$$ {_\fa}z_{p^\infty\ff}:=({_\fa}z_{p^n\ff})_{n\geq 1} \in \varprojlim_{K'} \co_{K'}[\frac{1}{p}]^\times\otimes_\bz\bz_p\simeq H^1(\co_{K}[\frac{1}{p}], \Lambda(1)).$$
\begin{lemma} For $\ff_\chi\mid\ff$ the image of ${_\fa}z_{p^\infty\ff}$ under the map
\begin{equation} H^1(\co_{K}[\frac{1}{p}], \Lambda(1))\to H^1(\co_{K}[\frac{1}{p}], \Lambda(\chi)(1))\label{projection}\end{equation}
induced by (\ref{lchidef}) coincides with $$\prod_{\fl\mid\ff,\fl\nmid p}(1-\chi(\fl)\Frob_{\fl}^{-1})\cdot {_\fa}\zeta(\chi).$$ Similarly, the image of $z_{p^\infty\ff}$ coincides with
$\prod_{\fl\mid\ff,\fl\nmid p}(1-\chi(\fl)\Frob_{\fl}^{-1})\cdot\zeta(\chi)$.
\label{zlemma}\end{lemma} 

\begin{proof}
By definition
\begin{align*}
{_{\fa}}\zeta_{K_{n}}(\chi)&=\trace_{K(\ff_{\chi}p^{n})/K_{n}}({_{\fa}}\zeta_{\ff_{\chi}p^{n}}(\chi))\\
&=\trace_{K(\ff_{\chi}p^{r+n})/K_{n}}({_{\fa}}z_{\ff_{\chi}p^{r+n}}\otimes t(\chi))
\end{align*}
for an integer $r$ with $\co_{K}^{\times}\rightarrow(\co_{K}/p^{r})^{\times}$ injective. So the image of 
${_{\fa}}z_{p^{\infty}\ff_{\chi}}$ under \eqref{projection} coincides with ${_{\fa}}\zeta(\chi)$. 
Note that \eqref{projection} factors through the map 
$$H^1(\co_{K}[\frac{1}{p}], \Lambda_{\ff}(1)) \to H^1(\co_{K}[\frac{1}{p}], \Lambda_{\ff_{\chi}}(1))$$
 induced by the projection $G_{p^{\infty}\ff}\twoheadrightarrow G_{p^{\infty}\ff_{\chi}}$ 
 where $\Lambda_{\fg}:=\bz_p[[\Gal(K(p^\infty\fg)/K)]]$ for an ideal $\fg\subset\co_K$.  
 It coincides with the norm map 
$$N_{K(\ff p^{\infty})/K(\ff_{\chi}p^{\infty})}: 
\varprojlim_{K'\subset K(\ff p^{\infty})} \co_{K'}[\frac{1}{p}]^\times\otimes_\bz\bz_p \to 
\varprojlim_{K' \subset K(\ff_{\chi} p^{\infty})} \co_{K'}[\frac{1}{p}]^\times\otimes_\bz\bz_{p}.$$ 
Recall the Euler system norm relation \cite{jlk}[Prop. 3.3 (2)] 
$$
N_{K(\ff p^{\infty})/K(\ff_{\chi}p^{\infty})}({_{\fa}}z_{\ff p^{\infty}})=\prod_{\fl\mid\ff,\fl\nmid p\ff_\chi}(1-\Frob_{\fl}^{-1})\cdot {_\fa}z_{\ff_{\chi}p^{\infty}} 
$$
and observe that $\chi(\Frob_{\fl}^{-1})=\chi(\fl)^{-1}\in\co$ acts on $\Lambda(\chi)$ via $\chi(\fl)$. Noting that $\chi(\fl)=0$ for $\fl\mid\ff_\chi$ the proof concludes. 
\end{proof}

\subsection{The main conjecture} \label{sec:mc} We shall also denote by $z_{p^\infty\ff}$ the image of $z_{p^\infty\ff}$ under the composition of (\ref{projection}) with the restriction map 
$$H^1(\co_{K}[\frac{1}{p}], \Lambda(\chi)(1))\otimes_{\Lambda_\co}Q(\Lambda_\co)\to H^1(\co_{K}[\frac{1}{p\ff}], \Lambda(\chi)(1))\otimes_{\Lambda_\co}Q(\Lambda_\co)$$
induced by the open immersion $j$. 
 
\begin{theorem} For $\ff_\chi\mid\!\ff$ there is an equality of invertible $\Lambda_\co$-submodules 
$$\Lambda_\co\cdot z_{p^\infty\ff} =  \mydet^{-1}_{\Lambda_\co} R\Gamma(\co_{K}[\frac{1}{p\ff}], \Lambda(\chi)(1))$$
of $\mydet^{-1}_{\Lambda_\co} R\Gamma(\co_{K}[\frac{1}{p\ff}], \Lambda(\chi)(1))\otimes_{\Lambda_\co}Q(\Lambda_\co)$.
\label{mc}\end{theorem} 

\begin{proof} By \cite{jlk}[Cor. 5.3]  there is an equality of invertible $\Lambda_\co$-submodules 
$$\Lambda_\co\cdot\prod_{\fl\mid\ff,\fl\nmid p}(1-\chi(\fl)\Frob_{\fl}^{-1})\cdot\zeta(\chi)=  \mydet^{-1}_{\Lambda_\co} R\Gamma(\co_{K}[\frac{1}{p\ff}], \Lambda(\chi)(1))$$
of $\mydet^{-1}_{\Lambda_\co} R\Gamma(\co_{K}[\frac{1}{p\ff}], \Lambda(\chi)(1))\otimes_{\Lambda_\co}Q(\Lambda_\co)$. Together with Lemma \ref{zlemma} this gives the result.
\end{proof}
\begin{remark} 
For primes $p\nmid |\co_K^\times|\cdot |G_{p^{\infty}\ff}^{tor}|$ the above main conjecture is equivalent  Rubin's main conjecture \cite{rubin88}, \cite{rubin91} (see \cite{jlk}[\S5.5]). 
\end{remark}

\subsection{Descent to Galois representation of Hecke characters} \label{sec:descent} Let $\varphi$ be an algebraic Hecke character of $K$ of infinity type $(-1,0)$ and with values in the number field $L$. For a prime number $p$ and place $\csp\mid p$ of $L$ let $V_{L_\csp}(\varphi)$ be the continuous $L_\csp$-linear $G_K$-representation associated to $\varphi$ as in section \ref{sec:hecke}. Choose a free, rank one $G_K$-invariant $\co:=\co_{L_{\csp}}$-submodule 
$$T_{\co_{L_{\csp}}}(\varphi)\subset V_{L_{\csp}}(\varphi)$$
and let
$$ \rho:G_{p^\infty\ff}\to \co^\times$$
denote the character giving the action of $G_K$ on $T_{\co_{L_{\csp}}}(\varphi)$. Here $\ff$ is any multiple of the conductor $\ff_{\varphi}$ of $\varphi$.
Choose a decomposition (\ref{deltadef}), i.e. a splitting $G_{p^\infty\ff}\to\Delta$ of the inclusion $\Delta:=G_{p^\infty\ff}^{tor}\subseteq G_{p^\infty\ff}$  and define a finite order character $\chi$ as the composite
$$ \chi:G_{p^\infty\ff}\to \Delta\xrightarrow{\rho^{-1}\vert_\Delta} \co^\times. $$

\begin{lemma} For primes $v\nmid p$ of $K$ we have 
$$(\ff_\chi)_v=(\ff_{\varphi})_v\ \left(= (\ff_{\rho})_v\right).$$
\label{conductor}\end{lemma}
\begin{proof} For $v\nmid p$ the image of the inertia subgroup $I_v$ in $G_{p^\infty\ff}$ is finite, hence lies in $\Delta$. By the definition of $\chi$ we have $\chi\vert_{I_v}=\rho^{-1}\vert_{I_v}$  and hence $(\ff_\chi)_v=(\ff_{\rho^{-1}})_v=(\ff_{\rho})_v$.  
 \end{proof}

\begin{remark} For $v\mid p$ the conductor of $\chi$ depends on the choice of a decomposition (\ref{deltadef}) and might differ from $(\ff_{\varphi})_v$ (in either direction).
\end{remark}

The following Lemma is a pared down generalization of \cite{flach03}[Lemma 5.7] from a one-variable to a two-variable Iwasawa algebra. Lemma 5.7 in \cite{flach03} computes the descent of a basis of the determinant of a perfect complex over a one-variable Iwasawa algebra. It might be possible to formulate a descent Lemma over a two-variable Iwasawa algebra in similar generality but we found it too confusing to do so for the simple application that we need.

\begin{lemma} Let $R$ be a two-dimensional regular local ring with fraction field $F$ and residue field $k$, $\Delta$ a perfect complex of $R$-modules and $\mathcal{L}\in H^1(\Delta)$ an element such that the following hold.
\begin{itemize} 
\item[1)] $H^1(\Delta)$ is $R$-torsion free and of $R$-rank one, $H^2(\Delta)$  is $R$-torsion and $H^i(\Delta)=0$ for $i\neq 1,2$. 
\item[2)] There is an equality of invertible $R$-submodules
$$ R\cdot\mathcal{L}= \mydet^{-1}_R\Delta$$
of $$H^1(\Delta)\otimes_RF\simeq \mydet^{-1}_R(\Delta)\otimes_RF.$$
\item[3)] The image $\bar{\mathcal{L}}$ of $\mathcal{L}$ under the natural map $H^1(\Delta)\to H^1(\Delta\otimes_R^{\mathbb L}k)$ is nonzero.
\item[4)] $H^0(\Delta\otimes_R^{\mathbb L}k)=0$
\end{itemize}
Then $H^i(\Delta\otimes_R^{\mathbb L}k)=0$ for $i\neq 1$, $\dim_k H^1(\Delta\otimes_R^{\mathbb L}k)=1$ and the image of $\mathcal{L}\otimes 1$ under the isomorphism
$$ (\mydet^{-1}_R\Delta)\otimes_Rk\simeq \mydet^{-1}_k(\Delta\otimes_R^{\mathbb L}k)\simeq H^1(\Delta\otimes_R^{\mathbb L}k)$$
coincides with $\bar{\mathcal L}$.
\label{descentlemma}\end{lemma}

\begin{proof} Let $a,b\in R$ be a system of parameters so that $k\simeq R/(a,b)$. The short exact sequences
$$ 0\to H^i(\Delta)\otimes_RR/a\to H^i(\Delta\otimes_R^{\mathbb L}R/a)\to H^{i+1}(\Delta)_a\to 0$$
and the fact that $H^1(\Delta)$ is torsion free show that $H^i(\Delta\otimes_R^{\mathbb L}R/a)=0$ for $i\neq 1,2$. The map in 3) factors
$$H^1(\Delta)\to H^1(\Delta)\otimes_RR/a\to H^1(\Delta\otimes_R^{\mathbb L}k)$$ 
and hence the image $\mathcal{L}_a$ of $\mathcal{L}$ in $H^1(\Delta)\otimes_RR/a$ is nonzero. By Nakayama's Lemma for the discrete valuation ring $R_{(a)}$ and the fact that $H^1(\Delta)_{(a)}$ is $R_{(a)}$-torsion free and of $R_{(a)}$-rank one the element $\mathcal{L}$ is a basis of $H^1(\Delta)_{(a)}$. From 2) we find $H^2(\Delta)_{(a)}=0$ and hence that $H^1(\Delta\otimes_R^{\mathbb L}R/a)_{(a)}$ has $R_{(a)}/a$-rank one.

We now have the perfect complex $\Delta':=\Delta\otimes_R^{\mathbb L}R/a$ over the discrete valuation ring $R':=R/a$ with uniformizer $b$ and fraction field  $F'=R_{(a)}/a$, cohomologically concentrated in degrees $1,2$, such that $H^2(\Delta')$ is $R'$-torsion and  $H^1(\Delta')$ is of rank one and $R'$-torsion free. This last claim follows from the isomorphism
$$ H^0(\Delta\otimes_R^{\mathbb L}k)\simeq H^0(\Delta'\otimes_{R'}^{\mathbb L}R'/b)\simeq H^1(\Delta')_b$$
and assumption 4).
 By assumption 3) the image of $\mathcal{L}_a\in H^1(\Delta')$ under
$$ H^1(\Delta')\to H^1(\Delta')\otimes_{R'}R'/b \to H^1(\Delta'\otimes_{R'}^{\mathbb L}R'/b)\simeq H^1(\Delta\otimes_R^{\mathbb L}k)$$
is nonzero, hence by Nakayama's Lemma $\mathcal L_a$ is an $R'$-basis of $H^1(\Delta')$ (as we already know that $H^1(\Delta')$ is free of rank one). 
From 2) we know that $\mathcal L\otimes 1$ is an $R'$-basis of 
$$ (\mydet^{-1}_R\Delta)\otimes_RR'\simeq \mydet^{-1}_{R'}(\Delta')\simeq H^1(\Delta')\otimes_{R'} \mydet^{-1}_{R'}H^2(\Delta') $$
and the image of $\mathcal L\otimes 1$ in
\begin{equation} (\mydet^{-1}_R\Delta)\otimes_RF'\simeq \mydet^{-1}_{F'}(H^1(\Delta)_{(a)}/a)\simeq H^1(\Delta')\otimes_{R'} F'\label{la}\end{equation}
coincides with $\mathcal L_a$. It follows that the $R'$-order ideal of the torsion module $H^2(\Delta')$ equals $R$ and hence that $H^2(\Delta')=0$. We deduce that
$$ 0=H^2(\Delta')\otimes_{R'}{R'}/b\simeq H^2(\Delta'\otimes_{R'}^{\mathbb L}R'/b)\simeq H^2(\Delta\otimes_R^{\mathbb L}k).$$
Since $\mathcal L_a$ is an $R'$-basis of 
$$H^1(\Delta')\simeq \mydet^{-1}_{R'}(\Delta')$$
the image $\bar{\mathcal L}$ of $\mathcal L_a$ in 
$$H^1(\Delta\otimes_R^{\mathbb L}k)\xleftarrow{\sim}H^1(\Delta')\otimes_{R'}R'/b\simeq \mydet^{-1}_{R'}(\Delta')\otimes_{R'}R'/b\simeq \mydet^{-1}_{R}(\Delta)\otimes_{R}k$$
is a $k$-basis. But $\bar{\mathcal L}$ is also the image of $\mathcal L\otimes 1\in\mydet^{-1}_{R}(\Delta)\otimes_{R}k$ as we saw in (\ref{la}). This concludes the proof.
\end{proof} 

\begin{remark} It follows from the proof of Lemma \ref{descentlemma} that $H^2(\Delta)=0$ and $H^1(\Delta)$ is free with basis $\mathcal L$ under the assumptions of Lemma \ref{descentlemma}.
\end{remark}

\begin{prop} Let $\gamma\in V_L(\varphi)$ be such that its image in $V_{L_{\csp}}(\varphi)$ is an $\co$-basis of $T_{\co_{L_{\csp}}}(\varphi)$ and let $\ff$ be a multiple of $\ff_{\varphi}$. Let $z_{p^\infty\ff}(\gamma)$ be the image of $z_{p^\infty\ff}$ under
\begin{align*} &H^1_{p^\infty\ff}(\bz_p(1))\xrightarrow{\gamma}H^1_{p^\infty\ff}(\bz_p(1))\otimes V_{L_{\csp}}(\varphi)\simeq H^1_{p^\infty\ff}(V_{L_{\csp}}(\varphi)(1))\\
\to & H^1(\co_K[\frac{1}{p}],V_{L_{\csp}}(\varphi)(1))\to H^1(\co_K[\frac{1}{p\ff}],V_{L_{\csp}}(\varphi)(1))
 \end{align*}
 and assume that $L_{p\ff}(\bar{\varphi},1)\neq 0$. Then  there is an equality of invertible $\co$-submodules 
$$ \co\cdot z_{p^\infty\ff}(\gamma)=  \mydet^{-1}_{\co} R\Gamma(\co_{K}[\frac{1}{p\ff}], T_{\co_{L_{\csp}}}(\varphi)(1))$$
of 
\begin{align*} \mydet^{-1}_{ {L_{\csp}}} R\Gamma(\co_{K}[\frac{1}{p\ff}], V_{L_{\csp}}(\varphi)(1)) \simeq \ &\mydet_{L_{\csp}} H^1(\co_{K}[\frac{1}{p\ff}], V_{L_{\csp}}(\varphi)(1))\\
\simeq \ & H^1(\co_{K}[\frac{1}{p\ff}], V_{L_{\csp}}(\varphi)(1)).\end{align*}
Moreover, the Selmer group $${\rm{Sel}}(K,T_{\co_{L_{\csp}}}(\varphi)(1))$$ is finite. 
\label{descent}\end{prop} 

\begin{proof} The character $\chi\rho$ is trivial on $\Delta$, hence induces a character $\chi\rho:\Gamma\to\co^\times$ and a ring homomorphism $\chi\rho:\bz_p[[\Gamma]]\to \co$. We obtain an induced ring homomorphism
$$ \kappa:=\id\otimes\chi\rho:\Lambda_\co\simeq \co\otimes_{\bz_p}\bz_p[[\Gamma]]\to\co$$
so that there is an $\co$-linear isomorphism of $G_{p^\infty\ff}$-modules
$$ \Lambda(\chi)\otimes_{\Lambda_\co,\kappa}\co\simeq T_{\co_{L_{\csp}}}(\varphi).$$
This induces an isomorphism of invertible $\co$-modules
\begin{align}&\left(\mydet^{-1}_{\Lambda_\co} R\Gamma(\co_{K}[\frac{1}{p\ff}], \Lambda(\chi)(1))\right)\otimes_{\Lambda_\co,\kappa}\co\notag\\ \simeq & 
\mydet^{-1}_{ \co}\left( R\Gamma(\co_{K}[\frac{1}{p\ff}], \Lambda(\chi)(1))\otimes^\bl_{\Lambda_\co,\kappa}\co\right) \notag\\
\simeq & \mydet^{-1}_{ \co}\left( R\Gamma(\co_{K}[\frac{1}{p\ff}], \Lambda(\chi)(1)\otimes_{\Lambda_\co,\kappa}\co)\right) \notag\\
\simeq & \mydet^{-1}_{ \co} R\Gamma(\co_{K}[\frac{1}{p\ff}], T_{\co_{L_{\csp}}}(\varphi)(1)).
\label{invertible}\end{align}
We apply Lemma \ref{descentlemma} in the following situation. Denote by $\fq$ the kernel of $\kappa$ and set
\begin{align*}  R:=\Lambda_{\co,\fq}; \quad \Delta:= R\Gamma(\co_{K}[\frac{1}{p\ff}], \Lambda(\chi)(1))_\fq;\quad  \mathcal L:=z_{p^\infty\ff}. \end{align*}
Then we have
\begin{align*}  k\simeq L_{\csp};\quad \Delta\otimes_R^{\mathbb L}k\simeq  R\Gamma(\co_{K}[\frac{1}{p\ff}], V_{\co_{L_{\csp}}}(\varphi)(1));\quad  \bar{\mathcal L}=z_{p^\infty\ff}(\gamma). \end{align*}
For assumption 1) of Lemma \ref{descentlemma} we refer to \cite{kato00}[\S15], assumption 2) is the $\fq$-localization of Theorem \ref{mc}, assumption 3) follows from $L_{p\ff}(\bar{\varphi},1)\neq 0$ and Prop. \ref{reciprocity} and assumption 4) 
$$H^0(\co_K[\frac{1}{p\ff}],V_{L_{\csp}}(\varphi)(1))=0$$
holds since the $G_K$-action on $V_{L_{\csp}}(\varphi)(1)$ is nontrivial. Lemma \ref{descentlemma} implies that the image of $z_{p^\infty\ff}\otimes 1$ in 
$$\mydet^{-1}_{L_{\csp}} R\Gamma(\co_{K}[\frac{1}{p\ff}], V_{\co_{L_{\csp}}}(\varphi)(1))\simeq H^1(\co_K[\frac{1}{p\ff}],V_{L_{\csp}}(\varphi)(1))$$
under the isomorphisms in (\ref{invertible}) coincides with $z_{p^\infty\ff}(\gamma)$. On the other hand it follows from Theorem \ref{mc} that $z_{p^\infty\ff}\otimes 1$ is an $ \co$-basis of the invertible $\co$-module (\ref{invertible}). This proves the first part of Prop. \ref{descent}.

Lemma \ref{descentlemma} also includes the vanishing of $H^2(\co_K[\frac{1}{p\ff}],V_{L_{\csp}}(\varphi)(1))$ and the fact that
$$ \dim_{L_\csp}H^1(\co_K[\frac{1}{p\ff}],V_{L_{\csp}}(\varphi)(1))=1.$$
Prop. \ref{reciprocity} then implies that the map in Prop. \ref{key}a) is nonzero and hence an isomorphism. The central vertical triangle in (\ref{wichtig}) shows that 
$$R\Gamma_f(K,V_{L_{\csp}}(\varphi)(1))=0.$$
By Lemma \ref{rgf} we have $ {\rm{Sel}}(K,T_{\co_{L_{\csp}}}(\varphi)(1))_{\bq_p}\simeq H^1_f(K,V_{L_{\csp}}(\varphi)(1))=0$.
\end{proof}

\begin{remark} 
In the situation of Prop. \ref{descent} the finiteness of the Mordell-Weil group is due to Coates and Wiles \cite{coates77}, Arthaud \cite{arthuad78} and Rubin \cite{rubin81}. For $L=K$ the finiteness of the Tate-Shafarevich group is due to Rubin \cite{rubin87} and his approach likely generalises. The above complementary approach via Lemma \ref{descentlemma} is essentially based on the arguments in \cite{kato00}[\S15].
\end{remark}

\subsection{Proof of Theorem \ref{main2}}\label{sec:ab} We have $K\subseteq \End_LH^0(A,\Omega_{A/K})=L$ and the CM type of $A$ is induced from $K$. Since $\varphi(\alpha)=\alpha$ for $\alpha\equiv 1 \mod \ff_\varphi$ the character $\varphi\circ N_{F/K}$ of the id\`ele group of $F:=K(\ff_\varphi)$ takes values in $K$, hence arises from an elliptic curve $E/F$ with CM by $\co_K$. By \cite{gos}[(4.4), (4.8)] and our assumption that $A/K$ is simple, $A$ is an isogeny factor of $B:=\Res_{F/K}E$. An argument as in the proof of Prop. \ref{bcomp} shows that the motivic structure associated to $\varphi$ in section \ref{sec:hecke} is isomorphic to the rational structure $H^1(A)$. We then verify the assumptions of Prop. \ref{keyabelian} for all primes $p$. We may choose
$$ T_{\co_{L_{\csp}}}(\varphi)(1)\simeq H^1(A_{\bar{K}},\bz_p(1))\simeq T_p({^t}\!A)$$
and hence the finiteness of $\Sha(A/K)_{p^\infty}$ and of $A(K)$ follow from Prop. \ref{descent}. The element $z:=z_{p^\infty\ff}(\gamma)$ of Prop. \ref{descent} satisfies the assumptions of Prop. \ref{keyabelian}. Indeed, assumption a) follows from Prop. \ref{descent} and assumptions b) and c) follow from Prop. \ref{reciprocity} if we choose $\gamma$ to be an $L$-basis of $H^1(A(\bc),\bq)$ which is also a $\co_{L_p}$-basis of  $H^1(A(\bc),\bz_p)$. The finiteness of $\Sha(A/K)$ then follows from the global 
formula for its cardinality given by Thm. \ref{main2}. This concludes the proof of Thm. \ref{main2}.

\bigskip

We next specialize Thm. \ref{main2} to the situation considered in \cite{bugross85}. Let $q\equiv 3\mod 4$ be a prime number and set $K=\bq(\sqrt{-q})$. Let  $H$ be the Hilbert class field of $K$ and $E/H$ an elliptic curve with $j$-invariant $j(\co_K)$ whose Hecke character is fixed by all $\sigma\in G:=\Gal(H/K)$. Set 
$$B=\Res_{H/K}E;\quad L=\End_K(B)_\bq.$$ 
It was shown in \cite{gross1}[Thm. 15.2.5] that $L$ is a CM field. Let $\varphi$ be the Serre-Tate character of $B/K$.
\begin{prop} Assume $L(\bar{\varphi},1)\neq 0$ and $\End_K(B)=\co_L$. Then Conjecture (12.3) of \cite{bugross85} holds true for $B/K$.
\label{bugr}\end{prop}

\begin{proof} The main work in this proof consists in matching the period of Def. \ref{perioddef} to that defined in \cite{bugross85}. The field $\bq(j(\co_K))$ has a unique real embedding as the class number $h=[H:K]$ is odd (see Remark \ref{r1}). Together with our given embedding $K\subseteq\bc$ this gives a distinguished embedding $\iota:K(j(\co_K))=H\subseteq \bc$. Following \cite{bugross85}[(10.2)] define $\gamma$ and $\omega$, uniquely up to $\co_K^\times$, by
$$ H_1(E^\iota(\bc),\bz)=\co_K\cdot\gamma;\quad H^0(\ce,\Omega_{\ce/\co_H})=\co_H\cdot\omega$$
and put
$$\Omega_\iota:=\int_\gamma\iota(\omega)\in \bc^\times/\co_K^\times.$$
For $\sigma\in G$ we have isomorphisms of $\co_K$-modules
\begin{equation} \Hom_H(E^\sigma,E)\otimes_{\co_K}H_1(E^{\iota\sigma}(\bc),\bz)\simeq H_1(E^{\iota}(\bc),\bz);\quad f\otimes\delta\mapsto f_*\delta.\label{isog}\end{equation}
Indeed this can be checked locally at each prime $p$ and for any $p$ there exists an isogeny $f:E^\sigma\to E$ of degree prime to $p$.
Since
\begin{equation} \co_L=\End_K(B)=\bigoplus_{\sigma\in G} \Hom_H(E^\sigma,E)\cdot\sigma\label{olstructure}\end{equation}
by \cite{gross1}[(15.1.5)] we see that 
$$ H_1(B(\bc),\bz)\simeq \bigoplus_{\sigma\in G} H_1(E^{\iota\sigma}(\bc),\bz) =\co_L\cdot\gamma $$
is free of rank one over $\co_L$ (see also \cite{bugross85}[proof of Prop. 10.12]).

The period $\Omega\in L_\br^\times$ and fractional $\co_L$-ideal $\fa(\Omega)$ of Def. \ref{perioddef} for $B/K$ satisfy
$$ H^0(\cb,\Omega_{\cb/\co_K})\otimes_{\co_K} \cd_{K/\bq}^{-1}=\Omega \cdot \fa(\Omega)\cdot \Hom_\bz(H_1(B(\bc),\bz),\bz)$$
under the Deligne period isomorphism $\mathrm{per}_B$, or equivalently
\begin{align*} H^0(\cb,\Omega_{\cb/\co_K})=&\Omega \cdot \fa(\Omega)\cdot \Hom_\bz(H_1(B(\bc),\bz),\bz)\otimes_{\co_K} \cd_{K/\bq}\\
=&\Omega \cdot \fa(\Omega)\cdot \Hom_{\co_K}(H_1(B(\bc),\bz),\co_K)
\end{align*}
under the $K_\br=\bc$-valued integration pairing. Define
\begin{equation}\gamma^*\in\Hom_{\co_K}(H_1(B(\bc),\bz),\co_K)\label{dual}\end{equation}
by
$$\gamma^*(\delta)=\begin{cases} c & \text{if $\delta=c\,\gamma\in H_1(E^\iota(\bc),\bz)$ with $c\in \co_K$} \\ 0 & \text{if $\delta\in H_1(E^{\iota\sigma}(\bc),\bz)$ with $\sigma\neq 1$.}  \end{cases}$$
Again by (\ref{isog}) and (\ref{olstructure}) the element $\gamma^*$ is a $\co_L$-basis of (\ref{dual}). However, $H^0(\cb,\Omega_{\cb/\co_K})$ need not be free over $\co_L$.
Following \cite{bugross85} let $M=LH$ be the composite field, an extension of degree $h^2$ of $K$. Since $H/K$ is unramified we have $\co_M\simeq \co_L\otimes_{\co_K}\co_H$. There is an isomorphism of $\co_M$-modules
\begin{equation} H^0(\cb,\Omega_{\cb/\co_K})\otimes_{\co_K}\co_H\simeq H^0(\cb_{\co_H},\Omega_{\cb_{\co_H}/\co_H})\simeq \bigoplus_{\sigma\in G} H^0(\ce^\sigma,\Omega_{\ce^\sigma/\co_H})\label{nd}\end{equation}
and we have isomorphisms of $\co_K$-modules
$$ \Hom_H(E^\sigma,E)\otimes_{\co_K}H^0(\ce,\Omega_{\ce/\co_H})\simeq H^0(\ce^\sigma,\Omega_{\ce^\sigma/\co_H});\quad f\otimes\eta\mapsto f^*\eta.$$
Again this can be checked locally at each prime $p$. From (\ref{olstructure}) we see that (\ref{nd}) is free of rank one over $\co_M$ with basis $\omega$. Extending scalars from $\co_K$ to $\co_H$ we then have an identity of free, rank one $\co_M$-modules
\begin{equation} H^0(\cb,\Omega_{\cb/\co_K})\otimes_{\co_K}\co_H=\Omega \cdot \fa(\Omega)\cdot \Hom_{\co_H}(H_1(B(\bc),\bz)\otimes_{\co_K}\co_H,\co_H)\notag\end{equation} 
under an $H_\br$-valued integration pairing $\mathrm{per}_{B_H}$. Since
$$\int_\delta\omega=0$$
for $\delta\in H_1(E^{\iota\sigma}(\bc),\bz), \ \sigma\neq 1$ we have
$$ \mathrm{per}_{B_{H_\iota}}(\omega)=\Omega_\iota\cdot\gamma^*$$
where $\Omega_\iota\in K_\br^\times\subset L_\br^\times\simeq M_\iota^\times$ with $M_\iota=M\otimes_{H,\iota}\bc$. We obtain an identity of invertible $\co_M$-submodules of $M_\iota$
$$ \co_M\cdot\Omega_\iota=\Omega \cdot \fa(\Omega)\otimes_{\co_L}\co_M.$$
The element $\Omega\in L_\br^\times$ is the period of some $K$-rational differential on $B$, for example we can take
$$\omega_B:=\sum_{\sigma\in G}\omega^\sigma\in H^0(\cb,\Omega_{\cb/\co_K}).$$
In order to compute $\fa(\Omega)$ recall that by \cite{bugross85}[(10.8)] there exist units $u_\sigma\in\co_M^\times$ for each $\sigma\in G$ such that
$$ u_\sigma\cdot \omega^\sigma=\omega; \quad u_{\sigma\tau}=u_\sigma^\tau\cdot u_\tau.$$
Indeed, the differential $\omega^\sigma$ is an $\co_H$-basis of $H^0(\ce^\sigma,\Omega_{\ce^\sigma/\co_H})$ and an $\co_M$-basis of (\ref{nd}), by the same reasoning as used above for $\omega$. We have 
$$\omega_B=\sum_{\sigma\in G}\omega^\sigma= \left(\sum_{\sigma\in G}u_\sigma^{-1}\right)\cdot\omega=:v^{-1}\cdot\omega$$ 
with $v\in M^\times$ and
\begin{equation}\Omega=v^{-1}\cdot\Omega_\iota.\label{periodidentity}\end{equation} 
Since $(v^\tau)=(u_\tau^{-1}v)=(v)$ for $\tau\in G$ and $M/L$ is unramified the principal $\co_M$-ideal $(v)$ descends to a fractional $\co_L$-ideal $\fa(\Omega)$. By \cite{bugross85}[Prop. 11.1] there is an element $m\in M$ so that
$$ \frac{L(\bar{\varphi}^\alpha,1)}{\Omega_\iota}=m^\alpha \quad \forall\alpha\in \Hom_{H,\iota}(M,\bc)=\Hom_{K,\iota}(L,\bc)$$ 
and such that the $\co_M$-ideal $(m)$ is $G$-invariant, hence descends to a fractional $\co_L$-ideal $\fm_B$. Conjecture (12.3) of \cite{bugross85} states that
\begin{equation} \fm_B=  \fg_{\Sha}\cdot\prod_v \fg_v/ \fg_{tor}^{1+c} \label{bugrossconj}\end{equation}
where
$$ \fg_{\Sha}:=|\Sha(B/K)|_L;\quad \fg_{tor}:=|B(K)|_L;\quad \fg_v:=|\Phi_v|_L$$
and $c$ denotes the complex conjugation of $L$.  On the other hand Theorem \ref{main2} states that
\begin{equation}t:=\frac{L(\bar{\varphi},1)}{\Omega}\in L^\times;\quad (t)= \fg_{\Sha}\cdot\prod_v \fg_v/ \fg_{tor}^{1+c}\cdot\fa(\Omega).\label{main3}\end{equation}
By (\ref{periodidentity}) we have $t=mv$  and $(t)=\fm_B\cdot\fa(\Omega)$ and hence we find that (\ref{bugrossconj}) and (\ref{main3}) are equivalent.
\end{proof}

\begin{remark}  By \cite{montgomery-rohrlich} the assumption $L(\bar{\varphi},1)\neq 0$ holds if $q\equiv 7\mod 8$ and $E=A(q)$ is the curve of conductor $(\sqrt{-q})$ studied in \cite{gross1}. The condition $\End_K(B)=\co_L$ may or may not hold. In \cite{bugross85}[Sec.3] examples are given for both maximal and non-maximal $\End_K(B)$.

\end{remark}

\subsection{Proof of Theorem \ref{main}} \label{sec:ell} The proof of Theorem \ref{main} amounts to the conjunction of Theorem \ref{main2} for all characters $\varphi_1,\dots,\varphi_r$ in Prop. \ref{CM}, restriction of coefficients from $L$ to $K$, and isogeny invariance of the $K$-equivariant BSD conjecture. We present these arguments in the following sequence of Lemmas. From now on the notation of Prop. \ref{CM} will be in effect. In particular $L$ denotes the semisimple algebra
$$ L=L_1\times\cdots\times L_r$$
with maximal order 
$$ \co_{L}:=\co_{L_1}\times\cdots\times\co_{L_r}.$$
For the Serre-Tate character $\bar{\varphi}=(\bar{\varphi}_1,\dots,\bar{\varphi}_r)$ of ${^t}\!B$ we denote by
$$ L(\bar{\varphi},s):=(\left(L(\bar{\varphi}_\tau,s)\right)_\tau\in\prod_{\tau\in\Hom(L,\bc)}\bc\simeq L_\bc$$
its $L_\bc$-valued L-function. For any prime number $p$ define a free, rank one $G_K$-invariant $\co_{L_{p}}$-submodule 
$$ T_p(\varphi):=\prod_{\csp\mid p}T_{\co_{L_{1,\csp}}}(\varphi_1)\times\cdots\times \prod_{\csp\mid p}T_{\co_{L_{r,\csp}}}(\varphi_r) 
\subseteq V_p(\varphi)$$
of $V_p(\varphi)$ introduced in Prop. \ref{bcomp}.

\begin{lemma} Let $\gamma\in V_{L}(\varphi)$ be such that its image in $V_p(\varphi)$ is a $\co_{L_p}$-basis of $T_p(\varphi)$ and let $\ff$ be a multiple of the conductor $\ff_B$ of $B/K$.  Let $z_{p^\infty\ff}(\gamma)$ be the image of $z_{p^\infty\ff}$ in $H^1(\co_{K}[\frac{1}{p\ff}], V_p(\varphi)(1))$ and assume that $L(\bar{\varphi},1)\neq 0$. Then there is an equality of invertible $\co_{L_p}$-submodules 
$$ \co_{L_p}\cdot z_{p^\infty\ff}(\gamma)=  \mydet^{-1}_{\co_{L_p}} R\Gamma(\co_{K}[\frac{1}{p\ff}], T_p(\varphi)(1))$$
of 
$$ \mydet^{-1}_{L_p} R\Gamma(\co_{K}[\frac{1}{p\ff}], V_p(\varphi)(1)) \simeq \mydet_{L_p} H^1(\co_{K}[\frac{1}{p\ff}], V_p(\varphi)(1)).$$
\label{descent2}\end{lemma} 

\begin{proof} This is immediate by combining Prop. \ref{descent} for $\varphi=\varphi_1,\dots,\varphi_r$ and all primes $\csp\mid p$ of the fields $L_i$.
\end{proof} 

\begin{lemma} Let $\gamma_1,\dots,\gamma_d$ be a $K$-basis of $V_{L}(\varphi)$ whose image in $V_p(\varphi)$ is a $\co_{K_p}$-basis of $T_p(\varphi)$ and let $\ff$ be a multiple of $\ff_B$. Assume that $L(\bar{\varphi},1)\neq 0$.  Then there is an equality of invertible $\co_{K_p}$-submodules 
$$ \co_{K_p}\cdot z_{p^\infty\ff}(\gamma_1)\wedge\cdots\wedge z_{p^\infty\ff}(\gamma_d)=  \mydet^{-1}_{\co_{K_p}} R\Gamma(\co_{K}[\frac{1}{p\ff}], T_p(\varphi)(1))$$
of 
$$ \mydet^{-1}_{K_p} R\Gamma(\co_{K}[\frac{1}{p\ff}], V_p(\varphi)(1)) \simeq \mydet_{K_p} H^1(\co_{K}[\frac{1}{p\ff}], V_p(\varphi)(1)).$$
\label{descent3}\end{lemma} 

\begin{proof} Since the map $\gamma\mapsto z_{p^\infty\ff}(\gamma)$ is $K$-linear it suffices to prove Lemma \ref{descent3} for one particular basis $\{\gamma_i\}$ satisfying its condition. Choosing $\gamma_i=b_i\cdot\gamma$ where $b_i$ is a $K$-basis of $L$ which is also an $\co_{K_p}$-basis of $\co_{L_p}$ and $\gamma$ is as in Lemma \ref{descent2} we deduce Lemma \ref{descent3} immediately from Lemma \ref{descent2}.
\end{proof} 

By Prop. \ref{bcomp} there is an isomorphism of $G_K$-representations $$V_p(\varphi)\simeq H^1_{et}(B\otimes_{K}\bar{\bq},\bq_p)$$ over $K_p$. The following is an analogue of Lemma \ref{descent3} where the  $G_K$-stable $\co_{K_p}$-lattice $T_p(\varphi)$ has been replaced by the $G_K$-stable $\co_{K_p}$-lattice $H^1(B\otimes_K{\bar{\bq}},\bz_p)$. Also recall the isomorphism $H^1(B\otimes_K{\bar{\bq}},\bz_p)(1)\simeq T_p({^t}B)$.

\begin{lemma} Let $\tilde{\gamma}_1,\dots,\tilde{\gamma}_d$ be a $K$-basis of $V_L(\varphi)\simeq H^1(B(\bc),\bq)$ whose image in $V_p(\varphi)\simeq H^1_{et}(B\otimes_{K}\bar{\bq},\bq_p)$ is a $\co_{K_p}$-basis of $H^1(B\otimes_K{\bar{\bq}},\bz_p)$ and let $\ff$ be a multiple of $\ff_B$. 
Assume that $L(\bar{\varphi},1)\neq 0$.  Then there is an equality of invertible $\co_{K_p}$-submodules 
$$ \co_{K_p}\cdot z_{p^\infty\ff}(\tilde{\gamma}_1)\wedge\cdots\wedge z_{p^\infty\ff}(\tilde{\gamma}_d)=  \mydet^{-1}_{\co_{K_p}} R\Gamma(\co_{K}[\frac{1}{p\ff}], T_p({^t}B))$$
of 
$$ \mydet^{-1}_{K_p} R\Gamma(\co_{K}[\frac{1}{p\ff}], V_p({^t}B)) \simeq \mydet_{K_p} H^1(\co_{K}[\frac{1}{p\ff}], V_p({^t}B)).$$
\label{descent4}\end{lemma} 

\begin{proof} 
The order $\End_K(B)$ is contained in the maximal order $\co_{L}$. By choosing $T_p(\varphi)$ to be the $\co_{L_p}$-span of $H^1(B\otimes_K{\bar{\bq}},\bz_p)$ inside  $V_p(\varphi)\simeq H^1_{et}(B\otimes_{K}\bar{\bq},\bq_p)$ we can assume that $H^1(B\otimes_K{\bar{\bq}},\bz_p)$ is contained in  $T_p(\varphi)$. Define a finite $\co_{K_p}[G_K]$-module $M$ by the exact sequence
\begin{equation} 0\to H^1(B\otimes_K{\bar{\bq}},\bz_p)(1)\simeq  T_p({^t}B) \to T_p(\varphi)(1)\to M\to 0.\label{exact}\end{equation}
This sequence induces an exact triangle in the derived category of $\co_{K_p}$-modules
$$ R\Gamma(\co_{K}[\frac{1}{p\ff}], T_p({^t}B))\to R\Gamma(\co_{K}[\frac{1}{p\ff}], T_p(\varphi)(1)) \to R\Gamma(\co_{K}[\frac{1}{p\ff}],M)\to$$
and an isomorphism of invertible $\co_{K_p}$-modules
\begin{align} \mydet^{-1}_{\co_{K_p}} R\Gamma(\co_{K}[\frac{1}{p\ff}], T_p(\varphi)(1))\simeq\  & \mydet^{-1}_{\co_{K_p}} R\Gamma(\co_{K}[\frac{1}{p\ff}], T_p({^t}B))\label{determinant}\\
& \otimes_{\co_{K_p}} \mydet^{-1}_{\co_{K_p}} R\Gamma(\co_{K}[\frac{1}{p\ff}],M).\notag\end{align}
The complex $R\Gamma(\co_{K}[\frac{1}{p\ff}],M)$ has finite cohomology groups and there is an equality of invertible $\co_{K_p}$-submodules
\begin{equation}\mydet^{-1}_{\co_{K_p}} R\Gamma(\co_{K}[\frac{1}{p\ff}],M)= \prod_i | H^i(\co_{K}[\frac{1}{p\ff}],M)|_{K_p}^{(-1)^i}=|M|^{-1}_{K_p}\label{euler}\end{equation}
of
$$\left(\mydet^{-1}_{\co_{K_p}} R\Gamma(\co_{K}[\frac{1}{p\ff}],M)\right)\otimes_{\co_{K_p}}K_p\simeq K_p.$$
Here $|N|_{K_p}$ denotes the order ideal of a finite $\co_{K_p}$-module $N$ and the last identity in (\ref{euler}) is Tate's formula for the Euler characteristic \cite{milduality}[Thm. I.5.1], or rather its equivariant generalization \cite{flach98}[Thm. 5.1]. If now $\gamma_1,\dots,\gamma_d$ is a basis as in Lemma \ref{descent3} we deduce from Lemma \ref{descent3}, (\ref{determinant}) and (\ref{euler})
\begin{equation} \co_{K_p}\cdot z_{p^\infty\ff}(\gamma_1)\wedge\cdots\wedge z_{p^\infty\ff}(\gamma_d)=  \mydet^{-1}_{\co_{K_p}} R\Gamma(\co_{K}[\frac{1}{p\ff}], T_p({^t}B))\cdot |M|^{-1}_{K_p}.\label{intermediate}\end{equation}
On the other hand, if $\tilde{\gamma}_1,\dots,\tilde{\gamma}_d$ is a basis as in Lemma \ref{descent4} the exact sequence (\ref{exact}) shows that
$$ \co_{K_p}\cdot\tilde{\gamma}_1\wedge\cdots\wedge\tilde{\gamma}_d=|M|_{K_p}\cdot \gamma_1\wedge\cdots\wedge \gamma_d$$
and $K$-linearity of $\gamma\mapsto z_{p^\infty\ff}(\gamma)$ gives
$$\co_{K_p}\cdot z_{p^\infty\ff}(\tilde{\gamma}_1)\wedge\cdots\wedge z_{p^\infty\ff}(\tilde{\gamma}_d)=|M|_{K_p} \cdot
z_{p^\infty\ff}(\gamma_1)\wedge\cdots\wedge z_{p^\infty\ff}(\gamma_d).$$
Comparing this last identity with (\ref{intermediate}) gives Lemma \ref{descent4}.
\end{proof} 

Since
$$ T_p({^t}\!B)\simeq\Ind_{G_F}^{G_K}T_p({^t}E)$$
Shapiro's Lemma gives a canonical isomorphism
$$R\Gamma(\co_{K}[\frac{1}{p\ff}], T_p({^t}B))\simeq R\Gamma(\co_{F}[\frac{1}{p\ff}], T_p({^t}E)).$$
Furthermore there are canonical isomorphisms
$$H^0(B,\Omega_{B/K})\simeq H^0(E,\Omega_{E/F})$$
and
$$H^1(B(\bc),\bz)\simeq H^1(E(\bc\otimes_{K}F),\bz)\simeq \prod_{v\mid\infty}H^1(E(F_v),\bz).$$
Lemma \ref{descent4} and Corollary \ref{breciprocity} can therefore be rewritten in terms of $E/F$ as follows.
 
\begin{lemma} Let $E/F$ be an elliptic curve as in Theorem \ref{main}, in particular assume that $L(\bar{\psi},1)\neq 0$. Let $\tilde{\gamma}_1,\dots,\tilde{\gamma}_d$ be a $K$-basis of $H^1(E(\bc),\bq)$ whose image in $H^1_{et}(E\otimes_{K}\bar{\bq},\bq_p)$ is a $\co_{K_p}$-basis of $H^1(E\otimes_K{\bar{\bq}},\bz_p)$ and let $\ff$ be a multiple of $\ff_B=N_{F/K}\ff_E\cdot D_{F/K}$. Put
$$ z:=z_{p^\infty\ff}(\tilde{\gamma}_1)\wedge\cdots\wedge z_{p^\infty\ff}(\tilde{\gamma}_d)\in \mydet_{K_p} H^1(\co_{F}[\frac{1}{p\ff}], V_p({^t}E)).$$ 
Then there is an equality of invertible $\co_{K_p}$-submodules 
$$ \co_{K_p}\cdot z=  \mydet^{-1}_{\co_{K_p}} R\Gamma(\co_{F}[\frac{1}{p\ff}], T_p({^t}E))$$
of 
$$ \mydet^{-1}_{K_p} R\Gamma(\co_{F}[\frac{1}{p\ff}], V_p({^t} E)) \simeq \mydet_{K_p} H^1(\co_{F}[\frac{1}{p\ff}], V_p({^t}E)).$$
Moreover 
$$z':=\mydet_{K_p}(\exp^*)(z)\in \mydet_{K_p} H^0(E_{F_p},\Omega_{E_{F_p}/F_p})$$
is an element of $\mydet_K H^0(E,\Omega_{E/F})$ such that
$$ \mydet_{K_\br}(\mathrm{per})\left(z' \right)=L_{p\ff}(\bar{\psi},1)\cdot \fa(z)\cdot\mydet_{\co_K}\left(\prod\limits_{v\mid\infty}H^1(E(F_v),\bz)\right)$$
where $\fa(z)$ is a fractional $\co_K$-ideal prime to $p$.
\label{descent5}\end{lemma} 

\begin{proof} This is immediate from Lemma \ref{descent4} and Corollary \ref{breciprocity}. Note that the assumption 
$$\co_{K_p}\cdot\tilde{\gamma}_1\wedge\cdots\wedge\tilde{\gamma}_d=\mydet_{\co_{K_p}}H^1(B\otimes_K{\bar{\bq}},\bz_p)$$
on the basis $\tilde{\gamma}_i$ translates into the fact that
$$\co_K\cdot\tilde{\gamma}_1\wedge\cdots\wedge\tilde{\gamma}_d=\fa(z)\cdot\mydet_{\co_K}\left(\prod\limits_{v\mid\infty}H^1(E(F_v),\bz)\right)$$
for some fractional $\co_K$-ideal $\fa(z)$ prime to $p$.
\end{proof}

\begin{proof} (of Theorem \ref{main}) It suffices to produce an element $z$ as in Prop. \ref{keyelliptic} for all prime numbers $p$. The content of Lemma \ref{descent5} is precisely that the element $z$ satisfies the assumptions of Prop. \ref{keyelliptic} for $S=\{v\mid p\ff\}$. Since it was shown in Prop. \ref{descent} that the $p$-primary part of $\Sha(E/F)$ (and of $E(F)$) is finite for any prime $p$, the finiteness of $\Sha(E/F)$ follows from the global formula for its cardinality given by Prop. \ref{keyelliptic}. This concludes the proof of Theorem \ref{main}.
\end{proof}


\end{document}